\documentclass[a4paper,reqno]{amsart}
\usepackage[all]{xy}           
\usepackage{amssymb}           
\usepackage{hyperref}
\usepackage{eucal}
\usepackage{enumerate}
\usepackage[usenames,dvipsnames]{color}

\numberwithin{equation}{section}

\newcommand{\abla}{\mathfrak{a}}
\newcommand{\ad}{\mathrm{ad}}
\newcommand{\bw}{\bigwedge}
\newcommand{\C}{\mathbb{C}}
\DeclareMathOperator{\Der}{\mathrm{Der}}
\DeclareMathOperator{\Diff}{\mathrm{Diff}}
\newcommand{\ddt}{\frac{\partial}{\partial t}}
\DeclareMathOperator{\End}{\mathrm{End}}

\newcommand{\heis}{\mathrm{H}}
\newcommand{\heisla}{\mathfrak{h}}
\newcommand{\fol}{\mathcal{F}}

\newcommand{\id}{\mathrm{Id}}

\newcommand{\K}{\mathbb{K}}
\let\ker\relax
\DeclareMathOperator{\ker}{\mathrm{Ker}}
\newcommand{\lap}{\Delta}
\newcommand{\lie}{\mathcal{L}}
\newcommand{\la}{\mathfrak{g}}
\newcommand{\N}{\mathbb{N}}
\newcommand{\Q}{\mathbb{Q}}
\newcommand{\pr}{\mathrm{pr}}
\newcommand{\R}{\mathbb{R}}
\newcommand{\tor}{\mathbb{T}}
\newcommand{\vf}{\mathfrak{X}}
\newcommand{\vol}{\mathrm{vol}}
\newcommand{\Z}{\mathbb{Z}}

\numberwithin{equation}{section}

\newtheorem{definition}{Definition}[section]
\newtheorem{lemma}[definition]{Lemma}
\newtheorem{theorem}[definition]{Theorem}
\newtheorem{proposition}[definition]{Proposition}
\newtheorem{corollary}[definition]{Corollary}
\newtheorem{remarkth}[definition]{Remark}

\newtheorem{example}[definition]{Example}
\newenvironment{remark}{\begin{remarkth}\upshape}{\hfill$\diamond$\end{remarkth}}
\newtheorem{question}{Question}

\title[Almost formality of quasi-Sasakian and Vaisman manifolds]{Almost formality of quasi-Sasakian and Vaisman manifolds with applications to nilmanifolds}

\author[B. Cappelletti-Montano]{Beniamino Cappelletti-Montano}
 \address{Dipartimento di Matematica e Informatica, Universit\`a degli Studi di
 Cagliari, Via Ospedale 72, 09124 Cagliari, Italy}
 \email{b.cappellettimontano@gmail.com}

\author[A. De Nicola]{Antonio De Nicola}
 \address{Dipartimento di Matematica, Universit\`a degli Studi di Salerno, Via Giovanni Paolo II 132, 84084 Fisciano, Italy}
 \email{antondenicola@gmail.com}

 \author[J.~C. Marrero]{Juan Carlos Marrero}
 \address{Unidad Asociada ULL-CSIC ``Geometr{\'\i}a Diferencial y Mec\'anica Geo\-m\'e\-tri\-ca''
Departamento de Matem\'aticas, Estad{\'\i}stica e Investigaci\'on Operativa, Facultad de Ciencias, Universidad de La Laguna, La Laguna, Tenerife, Spain}
 \email{jcmarrer@ull.edu.es}

\author[I. Yudin]{Ivan Yudin}
 \address{CMUC, Department of Mathematics, University of Coimbra, 3001-501 Coimbra, Portugal}
 \email{yudin@mat.uc.pt}



\thanks{This work was partially supported by CMUC -- UID/MAT/00324/2013, funded by the Portuguese
 Government through FCT/MEC and co-funded by the European Regional Development Fund through the Partnership Agreement PT2020 (A.D.N. and I.Y.),
 by MICINN (Spain) and European Union (Feder)  grant MTM 2015-64166-C2-2P (A.D.N. and J.C.M.)
 by Prin 2015 --- Real and Complex Manifolds; Geometry, Topology and Harmonic Analysis --- Italy and by GESTA and KASBA --- funded by Fondazione di Sardegna and Regione Autonoma della Sardegna (BCM)  --- Italy (B.C.M.), and by the exploratory research project in the frame of Programa Investigador FCT IF/00016/2013 (I.Y.). J.C.M. and B.C.M. acknowledge the Centre for Mathematics of the University of Coimbra in Portugal for its support and hospitality in a visit where a part of this work was done. The first and second authors are members of Indam - GNSAGA (Gruppo Nazionale per le Strutture Algebriche, Geometriche e le loro Applicazioni).}

\begin{document}

\begin{abstract}
We provide models
that are as close as possible to being formal
for a large class of compact manifolds that admit a
transversely K\"ahler structure, including Vaisman and quasi-Sasakian manifolds.
As an application we are able to classify the corresponding nilmanifolds.
\end{abstract}

\maketitle
\tableofcontents

\section{Introduction}

The existence of a K\"ahler or Sasakian structure on a compact manifold has
strong topological consequences. For instance, a compact K\"ahler manifold $M$
is formal \cite{DGMS} and satisfies the Hard Lefschetz Theorem \cite{hodge}. Recently it was discovered
that also compact Sasakian manifolds, which are considered as an odd dimensional
counterpart of K\"ahler manifolds, satisfy a Hard Lefschetz Theorem \cite{hlt}. Moreover,
in 2008 Tievsky \cite{tievsky} proved that in order to admit a Sasakian structure,
the de Rham algebra $\Omega^\ast (M)$ of a  compact  manifold $M$
has to be quasi-isomorphic as commutative differential graded algebra 
(CDGA for short) to an elementary Hirsch extension of
the basic cohomology algebra $H^*_B(M)$ of the canonical $1$-dimensional foliation defined
by the Reeb vector field. Both  the formality for K\"ahler
manifolds and the Tievsky model in the Sasakian case give topological
obstructions. For example, they imply that the only compact nilmanifolds that can be endowed
with a K\"ahler structure are the even dimensional tori (see \cite{hasegawa}), while the only Sasakian compact nilmanifolds are the
quotients of the generalized Heisenberg group $H(1, m)$ by  co-compact discrete
subgroups (see \cite{nilsasakian}).

Apart from K\"ahler metrics on a complex manifold $M$, an interesting class of Hermitian metrics on $M$ are the so-called locally conformal K\"ahler (l.c.K.) metrics (see \cite{DrOr}), that is, metrics which are conformally related with K\"ahler metrics in some open neighborhood of every point of $M$ (for a discussion of other interesting Hermitian metrics on $M$, which are locally conformal to special metrics, and the relation of them with l.c.K. structures, we remit to \cite{AnUg}).  
On the other hand, in the previous setting of compact nilmanifolds, it was conjectured by  Ugarte in \cite{Ug} that a compact nilmanifold endowed with a l.c.K. structure with non-zero Lee $1$-form is a compact quotient of $H(1, m) \times \mathbb{R}$. This conjecture is still open although some advances in the proof of it have been obtained by Bazzoni in \cite{bazzoni} for l.c.K. nilmanifolds with parallel Lee $1$-form, that is, for Vaisman nilmanifolds. However, not much is known about the topology of general Vaisman manifolds, which are
related both to K\"ahler and Sasakian geometry. A well-known fact is that the
difference between two consecutive Betti numbers $b_{k}(M) - b_{k-1}(M)$ is even
for each even integer $k\in\left\{1,\ldots,n\right\}$, where $\dim(M)=2n+2$.
Recently, the authors proved in \cite{hltvaisman} that a Hard Lefschetz Theorem holds for Vaisman
manifolds. This gives a topological obstruction stronger than the
aforementioned  property of the Betti numbers. The original motivation for the present
paper was to find a model for a compact Vaisman manifold. Indeed, we show that,
if $\eta$ denotes the anti-Lee $1$-form of a compact Vaisman manifold $M$, then the CDGA
\begin{equation}\label{modello}
\left( H^*_B(M,\fol) \otimes \bw \left\langle x,y \right\rangle, dx = 0, dy=
[d\eta]_B\right)
\end{equation}
is quasi-isomorphic to $\Omega^{\ast}(M)$, where $H^*_B(M,\fol)$ denotes the basic cohomology with respect to the  
flat $2$-dimensional foliation $\fol$ generated by the Lee and anti-Lee vector fields.
Recently, similar models are also considered in \cite{kasuya}.

In fact the above CDGA can be seen as an example of  a more general class that we call \emph{almost formal CDGAs}.
A CDGA $(B, d)$ is said to be \emph{almost formal} of  \emph{index}
$l$ if it is quasi-isomorphic to
the CDGA $\left(A\otimes \bw \left\langle y \right\rangle,dy=z\right)$, where $A$ is a
connected CDGA with the zero differential and $z\in A_2$ is a
closed homogeneous element satisfying  $z^l\not=0$, $z^{l+1}=0$.
Notice that compact Vaisman manifolds are in general not formal, as there are examples of compact Sasakian manifolds
that are not formal (see \cite{mariza}) and the product of a Sasakian manifold with the circle is Vaisman.
We show that there are many geometric realizations of almost formal CDGAs. Actually,
both Vaisman and quasi-Sasakian manifolds turn out to have almost formal models.
Further examples of almost formal manifolds are given by a generalization of Vaisman manifolds, which we call \emph{quasi-Vaisman}. 

A quasi-Vaisman manifold is a Hermitian manifold $(M,J,g)$ admitting a  closed $1$-form $\theta$ such that its metric dual is parallel,  holomorphic and
\begin{equation*}
d\Omega = \theta \wedge d\eta,
\end{equation*}
where $\Omega$ is the fundamental $2$-form of $M$ and $\eta=-\theta\circ J$. One proves that a  quasi-Vaisman manifold is Vaisman if and only if it is locally conformally symplectic  (l.c.s.) of the first kind (see \cite{lcs-Vai}; see also \cite{BaMa} for a recent discussion on l.c.s. manifolds of the first kind).

Having an almost formal model gives  strong constraints on the topology of the manifold. Indeed, we  are able to
 completely characterize almost formal nilmanifolds. In particular, we classify compact quasi-Vaisman and quasi-Sasakian nilmanifolds.  Namely, we prove that a compact quasi-Sasakian (resp. quasi-Vaisman) nilmanifold is a compact quotient of $H(1, m) \times \mathbb{R}^{2n}$ (resp. $H(1, m) \times \mathbb{R}^{2n+1}$), with $n \in \mathbb{N} \cup \{0\}$. In other words, we give a positive answer to a problem which is the counterpart to the Ugarte conjecture for compact quasi-Sasakian and quasi-Vaisman nilmanifolds.
 As a special case, we obtain a more conceptual  proof of the fact, recently proved in another way by Bazzoni \cite{bazzoni}, that a $(2n+2)$-dimensional compact nilmanifold $G/\Gamma$ admits a Vaisman structure if and only if $G$ is isomorphic to $H(1,n)\times\mathbb{R}$ as Lie groups.

We also provide examples of quasi-Vaisman solvmanifolds which are not Vaisman and are modelled on groups which are not allowed for quasi-Vaisman nilmanifold due to our classification.

In order to construct such examples we exploit the close relation between quasi-Vaisman and quasi-Sasakian manifolds. In some sense, Vaisman manifolds is to Sasakian geometry as quasi-Vaisman manifolds is to quasi-Sasakian geometry. Indeed we prove that the mapping torus of a quasi-Sasakian manifold is quasi-Vaisman. In turn, the mapping torus of a quasi-Vaisman manifold carries a canonical quasi-Sasakian structure. The interplays between these concepts motivates us to construct new examples of quasi-Sasakian manifolds.  In particular, we provide explicit examples of quasi-Sasakian solvmanifolds and the corresponding almost formal models.

Though our truly motivation was the study of Vaisman manifolds, the methodology for finding
the model \eqref{modello} for (quasi-)Vaisman manifolds can be applied in  more
general contexts. In fact, what is needed is just an infinitesimal action $f: \mathfrak{g}\rightarrow
\mathfrak{X}(M)$ of an  abelian Lie algebra on a Riemannian manifold $(M,g)$
such that every $f(a)$ is a Killing vector field, and the existence of an algebraic connection
$\chi : {\mathfrak{g}}^{\ast}\rightarrow \Omega^{1}(M)$ for $f$. This occurs for
Vaisman, and more in general for quasi-Vaisman manifolds, where we have the
$2$-dimensional Riemannian foliation defined by the commuting Lee and anti-Lee
vector fields. We also obtain a model for a mapping torus of a Riemannian manifold $(M,g)$ induced by an isometry in terms of the invariant forms of $M$.
The same methods could be applied also for a large class of
manifolds, such as locally conformal hyperK\"ahler manifold with parallel Lee form \cite{lchk}, which are foliated by a Riemannian $4$-dimensional flat foliation defined by the Lee vector field and its multiplication by the complex structures, or normal metric contact pairs \cite{pairs} (known also as Hermitian bicontact manifolds \cite{bicontact}), where one has $2$ commuting Killing Reeb vector fields, or $\mathcal S$-manifolds, where we have $s$ mutually commuting Killing vector fields \cite{blairs}.


\section{Elements of rational homotopy theory}
The aim of this section is to give an overview of results in rational homotopy
theory that we use in this article. The main references  for this section
are~\cite{felix_book} and~\cite{greub3}.

\subsection{Basics on CDGAs}\label{CDGAs}
Throughout this section $\K$ denotes a field of characteristic $0$. The most relevant
cases for geometry are when $\K$ is the field of rational numbers $\Q$, the field of real
numbers $\R$, or the field of complex numbers $\C$. In this paper only the cases
$\K=\R$, $\C$ will appear.

Let $V$, $W$ be  graded vector spaces over $\K$. We say that a $\K$-linear map $f\colon V \to W$
is homogeneous of degree $k$ if $f(V_m) \subset W_{m+k}$ for all $m\in \N$.

A \emph{derivation} of degree $k$ of a graded algebra  $A$ over $\K$ is a homogeneous map $D\colon A\to A$
where we consider $A$ as a graded vector space, such that
\begin{equation*}
D(a a') = D(a) a' + (-1)^{kl} a D(a')
\end{equation*}
for any $a\in A_l$.
We will denote the set of derivations of degree $k$ on $A$
by $\Der_k(A)$.

A \emph{commutative differential graded algebra } $\left( A,d \right)$
(CDGA for short) over
$\K$
is a graded algebra $A = \bigoplus_{k\ge 0} A_k$ over $\K$ such that for all $x\in A_k$
and $y\in A_l$ we have
\[
x y = \left( -1 \right)^{kl} y x,
\]
endowed with a differential $d$, that is $d\colon A \to A$ is a derivation of degree 1,
such that $d^2 =0$.
A \emph{morphism of CDGAs} is a morphism of graded algebras $f:A\to B$ such that
$d\circ f= f \circ d$.
An example of commutative differential graded algebra over $\R$ is given by the de Rham
complex $\left( \Omega^*\left( M \right), d \right)$ of differential forms on a
smooth manifold $M$, with the multiplication given by the wedge product.

The \emph{graded commutator} of a derivation $A$ of degree $k$ and a derivation $B$ of degree $l$ is defined by
\[
[A,B]= AB- \left( -1 \right)^{kl} BA.
\]

 A CDGA $\left( A,d \right)$ is
 \emph{directly quasi-isomorphic} to  a CDGA $\left( B,d
\right)$ if there is a
morphism of CDGAs $f\colon A\to B$ such that
$$
H^k\left( f \right)\colon H^k\left( A \right) \to H^k\left( B \right)
$$
 are isomorphisms for all $k\ge 0$.
Two CDGAs $\left( A,d \right)$ and $\left( B,d
\right)$ are \emph{quasi-isomorphic} if there is a chain of CDGAs $A=A_0$, $A_1$, \dots, $A_r = B$, such that either $A_j$ is directly
quasi-isomorphic to $A_{j+1}$ or
$A_{j+1}$ is directly quasi-isomorphic to  $A_j$ for every $0\le j\le r-1$.
We say that a CDGA $(A,d)$ is \emph{formal} if it is quasi-isomorphic to
$(H^*(A),0)$.

Given a graded vector space $V = \bigoplus_{k\ge 0} V_k$, we denote by $\bigwedge V$ the free
commutative graded algebra  generated by $V$. Recall that $\bigwedge V$, considered as an associative
algebra,  is the tensor product of the exterior
algebra constructed on $\bigoplus_{k\ge 0} V_{2k+1}$
with the symmetric algebra on $\bigoplus_{k\ge 0} V_{2k}$.
Thus every element in $\bw  V$ can be written as a sum of the elements of
the form
\begin{equation*}
v := v_1^{a_1} \dots v_m^{a_m}
\end{equation*}
where $v_j \in V_{k_j}$ and $a_j\in \N$
if $k_j$ is even while $a_j=1$ if $k_j$ is odd. We define the degree of $v$ to be
\begin{equation*}
a_1 k_1 + \dots + a_m k_m.
\end{equation*}

\subsection{Hirsch extensions of a CDGA}\label{V-CDGAs}
In order to define Hirsch extensions of a CDGA,
we will need the following notion which was extensively studied in Chapter~III
of \cite{greub3}.
\begin{definition}
Let $V$ be a graded vector space. A $V$-CDGA is a triple
$(A,d,f)$, where $(A,d)$ is a CDGA and $f\colon V\to A$ is a homomorphism of
graded vector spaces of degree $1$, such that $d\circ f =0$.
\end{definition}
\begin{remark}
In \cite{greub3}, $V$-CDGAs were called $V$-differential algebras and the  expression ``$V$-differential
algebra'' was abbreviated with ``$(V,\delta)$-algebra''.
\end{remark}

Suppose $( A,d,f)$ is a $V$-CDGA. Then $A\otimes \bw  V$ is a commutative graded algebra with the
multiplication defined by
\begin{equation*}
a\otimes v \cdot a' \otimes v' := \left( -1 \right)^{kl}(aa')\otimes vv'
\end{equation*}
where $v\in (\bw  V)_k$ and $a'\in A_l$.
We define the differential $d_f$ on $A\otimes \bw  V$ by
\begin{equation*}
d_f (a \otimes 1 ) := da \otimes 1,\ d_f (1 \otimes v) := f(v) \otimes 1,
\end{equation*}
where $v\in V$ and we extend $d_f$ to $A\otimes \bw  V$ by using Leibniz
rule.
We say that $\left(A\otimes \bw  V, d_f\right)$ is an \emph{Hirsch extension} of
$(A,d)$ by $V$ (along $f$).
If $V$ is a graded  vector space of dimension $m$ and $y_1$,\dots, $y_m$
is a homogeneous basis of $V$, then we will often specify $\left(A\otimes \bw
V, d_f \right)$ as
\begin{equation*}
\left(A\otimes \mbox{$\bw $}\,\left\langle y_1,\dots, y_m \right\rangle,\, dy_1
=f(y_1), \dots,\, dy_m = f(y_m)\right).
\end{equation*}
\begin{remark}
The underlying complex of $\left( A\otimes \bw V, d_f \right)$
was called a \emph{Koszul complex} in \cite{greub3}.
\end{remark}
Let $( A,d,f_A)$ and  $( B,d,f_B)$ be $V$-CDGAs. A homomorphism $h\colon
A\to B$ of CDGAs is called a \emph{homomorphism of $V$-CDGAs} if $ h\circ f_A
= f_B $.
Given a homomorphism $h\colon ( A,d,f_A) \to ( B,d, f_B)$ of $V$-CDGAs, we define
\begin{equation*}
\tilde{h} \colon \left(A\otimes \bw V,\, d_{f_A} \right) \to \left(B\otimes \bw V,\, d_{f_B}\right)
\end{equation*}
by
\begin{equation*}
\tilde{h}(a\otimes v_1 \wedge \dots\wedge v_k ) = h(a) \otimes v_1 \wedge \dots
\wedge v_k.
\end{equation*}
It is clear that $\tilde{h}$ is a homomorphism of CDGAs.
\begin{remark}\label{iso-iso}
Note that if $h$ is an isomorphism of $V$-CDGAs then, clearly, $\tilde{h}$ is an isomorphism of CDGAs, as $\tilde{h}^{-1}=\widetilde{h^{-1}}$.
\end{remark}

We will say that a homomorphism $h$ of $V$-CDGAs is a \emph{quasi-isomorphism} of
$V$-CDGAs if $h$ is a quasi-isomorphism of underlying CDGAs.
The following claim is a specialization of Proposition~4.3 in
\cite{sullivan}.
\begin{proposition}
\label{hf}
If $h$ is a quasi-isomorphism of $V$-CDGAs then $\tilde{h}$ is a quasi-isomorphism
of CDGAs.
\end{proposition}
Two $V$-CDGAs $(A, d, f_A)$ and $(B, d, f_B)$ are \emph{quasi-isomorphic} if
there is a sequence of $V$-CDGAs
\begin{equation*}
( A,d, f_A ) = ( A_0,d, f_0) ,\
( A_1,d, f_1), \dots, ( A_r,d, f_r) = (  B,d , f_B ),
\end{equation*}
 such that for
every $1\le k\le r$ either there is a quasi-isomorphism of $V$-CDGAs
\begin{equation*}
h_k\colon (A_k,d,f_k) \to ( A_{k-1},d,f_{k-1})
\end{equation*}
or a quasi-isomorphism of $V$-CDGAs
\begin{equation*}
h_k\colon (A_{k-1},d,f_{k-1}) \to ( A_{k},d,f_{k}) .
\end{equation*}

Given a $V$-CDGA $( A,d,f)$, we will denote by $f^\#$ the composite
\begin{equation*}
V \xrightarrow{f} \ker( d\colon A\to A) \to H^*(A).
\end{equation*}

It is clear that if two $V$-CDGAs are quasi-isomorphic as $V$-CDGAs then they are
also quasi-isomorphic as CDGAs. The following converse result is Proposition~XI
in Chapter~III of \cite{greub3}.
\begin{proposition}
\label{Vqiso}
Let $V$ be a graded vector space such that $V_{2k}=0$ for all $k\ge 0$.
Suppose $( A,d, f_A)$ and $( B,d, f_B)$ are $V$-CDGAs that are
quasi-isomorphic as CDGAs. Let $\gamma\colon H^*(A) \to H^*(B) $ be the
isomorphism  induced by a chain of quasi-isomorphisms connecting $(A,d)$ and
$(B,d)$. If ${\gamma\circ f_A^\# = f_B^\#}$, then there is a chain of quasi-isomorphisms of $V$-CDGAs connecting
$(A,d)$ and $(B,d)$ such that the resulting induced isomorphism from
$H^*(A)$ to $H^*(B)$ is equal to $\gamma$.
\end{proposition}
\begin{remark}
The condition $V_{2k}=0$ for $k\ge 0$ is most surely redundant, but it was
difficult to find an appropriate reference, and as we will use only the case
when $V$ is concentrated in degree $1$, we decided to not pursue the case of
general $V$ in this article.
\end{remark}
Now, combining Proposition~\ref{hf} and Proposition~\ref{Vqiso}, we get
\begin{theorem}\label{passage}
Let $V$ be a graded vector space concentrated in odd degrees.
Suppose $( A,d, f_A)$ and $( B,d, f_B)$ are $V$-CDGAs that are
quasi-isomorphic as CDGAs.  Denote by $\gamma\colon H^*(A) \to
H^*(B)$ the isomorphism in cohomology induced by the chain of
quasi-isomorphisms between $(A,d)$ and $(B,d)$.
If $\gamma \circ f_A^\# = f_B^\#$ then $\left(A\otimes \bw V, d_{f_A}\right)$ and
$\left(B\otimes \bw V, d_{f_B}\right)$ are quasi-isomorphic.
\end{theorem}
Theorem~\ref{passage} becomes very useful in the case of a formal CDGA.
\begin{corollary}
\label{substitution}
Suppose $V$ is a graded vector space concentrated in odd degrees and $(
A,d,f)$ is a $V$-CDGA such that $(A,d)$ is a formal CDGA. Then $\left(A\otimes \bw
V, d_f \right)$ is quasi-isomorphic to $\left(H^*(A)\otimes \bw V, d_{f^\#}\right)$.
\end{corollary}

\begin{definition}
A CDGA $(A,d)$ is called a \emph{finitely generated Sullivan algebra}, if there
is a sequence of CDGAs $(A_0,d_0) \cong (\K,0)$, \dots, $(A_n,d_n)
\cong(A,d)$,
such that $(A_{j+1}, d_{j+1})$ is an elementary extension of $(A_j,d_j)$ by a
finitely dimensional graded vector space.
\end{definition}
Note that the general definition of a Sullivan algebra can be given in a similar
manner, but it involves technicalities on infinite ordinals.
We will use the following
\begin{theorem}\label{sullivan}
\begin{enumerate}[(a)]
\item Let $(A,d)$ be a CDGA and $(B,d)$ a Sullivan algebra quasi-isomorphic to
 $(A,d)$. Then there is a  quasi-isomorphism $(B,d) \to (A,d)$.
\item Let $(A,d)$ be a CDGA. Then there is a Sullivan algebra $(B,d)$, which is quasi-isomorphic to $(A,d)$.
\end{enumerate}
\end{theorem}
The formal definition of a minimal Sullivan algebra can be found
in~\cite{felix_book}. In the finitely generated case the definition can be
phrased as follows.
\begin{definition}
 A finitely generated Sullivan algebra $(A,d)$ is called \emph{minimal (Sullivan)} if for
any finitely generated Sullivan algebra $(B,d)$ which is quasi-isomorphic to
$(A,d)$ and for each integer $k$, we have $\dim A_k \le \dim B_k$.
\end{definition}
For every CDGA $(A,d)$ there is a unique minimal (Sullivan) algebra $(B,d)$, which is quasi-isomorphic
to $(A,d)$. We will call $(B,d)$ the \emph{minimal model} of $(A,d)$.

%

Given a smooth manifold $M$, the de Rham complex  of differential forms $\left( \Omega^*\left( M \right),d  \right)$ endowed with the wedge product  is a CDGA.
We say that a CDGA $\left( A,d \right)$ is a  \emph{model} for a manifold $M$ if $\left( A,d \right)$ is quasi-isomorphic to
$\left( \Omega^*\left( M \right),d  \right)$. The minimal model of $\left( \Omega^*\left( M \right),d  \right)$ will be also called the  \emph{minimal model} of $M$.

\subsection{Operation of a Lie algebra in a CDGA}
\label{sec:operation}

To motivate the definition of an operation of a Lie algebra in a CDGA, we will start by considering the right action
of a Lie group $G$ on a smooth manifold $M$.
Then we have a homomorphism of topological groups
\begin{equation*}
F \colon G \to \Diff(M).
\end{equation*}
By passing to tangent spaces at the neutral element, we get a homomorphism of Lie algebras
\begin{equation*}
f \colon \la \to \vf(M),
\end{equation*}
where $f:= T_eF$  and $\vf(M)$ is the Lie algebra of vector fields on $M$.
This motivates the following definition.
\begin{definition}
Let $M$ be a smooth manifold and $\la$ a Lie algebra. An \emph{(infinitesimal)
right action of $\la$ on $M$} is a homomorphism of Lie algebras $f \colon
\la \to \vf (M)$.
\end{definition}
For every point $p\in M$, we denote by $\mathrm{ev}_p$ the map from
$\vf(M)$ to $T_p M$ defined by evaluating a vector field $X$ at the point
$p$.

We say that an action $f\colon \la \to \vf(M)$ is \emph{free} if for every point
$p\in M$
the composition
\begin{equation*}
\la \xrightarrow{f} \vf(M) \xrightarrow{\mathrm{ev}_p}  T_pM
\end{equation*}
is injective.
Given a free action $f\colon \la \to \vf(M)$
we construct a map of vector bundles
\begin{align*}
\hat{f} \colon  M \times \la &\to TM\\
(p,a) & \mapsto (p,\mathrm{ev}_p \circ f(a)).
\end{align*}
The image of $\hat{f}$ generates an
integrable
distribution $\mathcal{D}_f\subset TM$ of rank $k=\dim \la$. We denote the corresponding foliation by
$\fol_f$.

Let $f\colon \la \to \vf(M)$ be an action of a Lie algebra  $\la$
on  a smooth manifold $M$. Then for every $a\in \la$, we have the usual derivations
$i_{f}(a)=i_{f(a)}$ and $\lie_{f}(a)=\lie_{f(a)}$ on $\Omega^*(M)$
of degree $-1$ and $0$, respectively.
Thus we get two linear maps
\begin{equation*}
i_f \colon \la \to \Der_{-1}(\Omega^*(M)),\quad  \lie_f \colon \la \to
\Der_0(\Omega^*(M)).
\end{equation*}
These maps have the following properties
\begin{equation}\label{action}
\begin{aligned}
\lie_f([a,b]) &= \left[ \lie_f(a), \lie_f(b) \right]\\
\lie_f(a) &= \left[ i_f(a),d \right]\\
i_f(a)^2 &=0\\
i_f([a,b]) &= [\lie_f(a), i_f(b)]\\
\end{aligned}
\end{equation}
This motivates the following definition.
\begin{definition}
Let $\la$ be a Lie algebra and $(A,d)$ a CDGA. We say that a linear map $i\colon \la
\to \Der_{-1}(A)$ is an \emph{operation} of $\la$ in $(A,d)$ if for $\lie\colon \la \to
\Der_0(A)$ defined by $\lie(a) = \left[ i(a),d \right]$, the
equations~\eqref{action} hold upon erasing subscript $f$.
\end{definition}
\begin{remark}
Given an operation $i$, in the sequel we will write $i_a$ and $\lie_a$ instead of $i(a)$ and $\lie(a)$, respectively.
\end{remark}
Note that the third equation in \eqref{action} can be stated in a stronger form.
\begin{lemma}\label{anticommuting}
Let $i\colon \la \to \Der_{-1}(A)$ be an operation in a CDGA $(A,d)$. Then for
every $a$, $b\in \la$, we have
\begin{equation*}
[i_a, i_b] =0.
\end{equation*}
\end{lemma}
\begin{proof}
We have $i_{a+b}^2=0$. As $i$ is a linear map, this implies that
$(i_a +i_b)^2=0$. Using that $i_a^2 = i_b^2=0$, we get
$i_a i_b + i_b i_a=0$.
\end{proof}

Let $i\colon \la \to \Der_{-1}(A)$ be an operation in a CDGA $(A,d)$ and
$\lie = \left[ i, d \right]$. Then we define the CDGAs $(A_{\lie},d)$ and
$(A_{i,\lie},d)$ by
\begin{equation*}
A_\lie := \left\{\, a \in A \,\middle|\, \lie_x a =0, \forall x \in \la
\right\}; \quad A_{i,\lie}:= \left\{\, a \in A \,\middle|\, \lie_x a=i_x a =0,
\forall x \in \la \right\}.
\end{equation*}
Specializing to the case of an operation arising from an action $f\colon \la
\to \vf(M)$, we recover the invariant de Rham complex
$\Omega^*_{\lie_f}(M)$ and the basic de Rham complex
$\Omega^*_{i_f,\lie_f}(M) = \Omega^*_B(M,\fol_f)$.
\begin{theorem}\label{invariant}
Let $f\colon \la \to \vf(M)$ be an action on a compact Riemannian manifold
$(M,g)$.
Suppose $f(x)$ is a Killing vector field for every $x\in \la$. Then the
inclusion $h\colon \Omega^*_{\lie_f} (M) \hookrightarrow \Omega^*(M)$ induces an
isomorphism in cohomology, in other words, $h$ is a quasi-isomorphism of
CDGAs.
\end{theorem}
\begin{proof}
The result follows from \cite[Theorem~3.4]{goldberg-formula} and
\cite[Theorem~3.5]{goldberg-formula}.
\end{proof}

\subsection{Algebraic connection}\label{connection}
Let $\la$ be a finite dimensional Lie algebra.
We denote by $\ad^*$ the coadjoint representation of $\la$
\begin{equation*}
\ad^* \colon \la  \to \End(\la^*)
\end{equation*}
defined by
\begin{equation*}
\ad_a^*(\alpha)(b) = -\alpha([a,b]),
\end{equation*}
for $a,b\in \la$ and $\alpha \in \la^*$.
Let $i\colon \la \to \Der_{-1}(A)$ be an operation of a Lie algebra $\la$ in a CDGA $(A,d)$.
As usually, $\lie = [i,d]$.
Following Chapter VIII of \cite{greub3}, we say that
\[
\chi\colon \la^* \to A_1
\]
is an \emph{algebraic connection} for $i$ if
\begin{equation}\label{eq:connection}
\begin{aligned}
i_a (\chi(\alpha)) & = \alpha(a),\quad a \in \la,\ \alpha \in \la^*;\\
\lie_a \circ \chi & = \chi \circ \ad^*_a,\ a \in \la.
\end{aligned}
\end{equation}
\begin{remark}\label{standard-connection}
Note that for $A=\Omega^*(M)$, an algebraic connection
$\chi: {\mathfrak g}^* \to \Omega^1(M)$ corresponds to  a vector bundle
map $\widehat{\chi}: M \times {\mathfrak g}^{*} \to
T^*M$, given by $\widehat{\chi}\left( p, \alpha \right) = \chi(\alpha)_p$, from the
trivial vector bundle $M \times {\mathfrak g}^*$ to the cotangent bundle $T^*M$ of
$M$. So, we can consider the map $pr_2 \circ \widehat{\chi}^*: TM \to {\mathfrak g}$. This map is a standard connection for the infinitesimal action $f: {\mathfrak g} \to {\mathfrak X}(M)$. Define  a subbundle of the tangent bundle by
\[
x \in M \to H(x) = \{v \in T_xM \,|\, (pr_2 \circ \widehat{\chi}^*)(v) = 0\}.
\]
Then $H$ is an Ehresmann connection on $TM$.
In fact, the first condition in \eqref{eq:connection} implies that $TM =   
\mathcal{D}_f \oplus H$. Moreover, if we also denote by $\widehat{\chi}^*:
{\mathfrak X}(M) \to \Gamma(M \times {\mathfrak g})$ the corresponding morphism of
$C^{\infty}(M)$-modules between ${\mathfrak X}(M)$ and the space of sections of the
trivial vector bundle $M \times {\mathfrak g} \to M$ then $\widehat{\chi}^*$ is equivariant, that is,
\[
\widehat{\chi}^* \circ \lie_a = \ad_a \circ \widehat{\chi}^*, \; \; a \in \la.
\]
Here, $ad$ is the natural extension of the adjoint action of ${\mathfrak g}$ to the space of sections $\Gamma(M \times {\mathfrak g})$.
\end{remark}
Now, suppose that $(M, g)$ is a Riemannian manifold and that $f: {\mathfrak g} \to {\mathfrak X}(M)$ is a free action of ${\mathfrak g}$ on $M$. Denote by $\flat_g: TM \to T^*M$ the vector bundle isomorphism induced by $g$. Then, in the trivial vector bundle $M \times {\mathfrak g} \to M$, we can consider the bundle metric $\langle \cdot, \cdot \rangle$ given by
\[
\langle \xi, \eta \rangle = g(f(\xi), f(\eta))
\]
for $\xi, \eta \in \Gamma(M \times {\mathfrak g})$.
It is clear that $\langle \cdot, \cdot \rangle$ induces an isomorphism between the vector bundle $M \times {\mathfrak g} \to M$ and its dual bundle $M \times {\mathfrak g}^* \to M$
\[
\flat_{\langle \cdot, \cdot \rangle }: M \times {\mathfrak g} \to M \times {\mathfrak g}^*.
\]
We will denote by
\[
\sharp_{\langle \cdot, \cdot \rangle }:  M \times {\mathfrak g}^* \to M \times {\mathfrak g}
\]
the inverse morphism of $\flat_{\langle \cdot, \cdot \rangle }$.


\begin{theorem}\label{thm:connection}
Let $(M,g)$ be a Riemannian manifold, and $\la$ a Lie algebra. Suppose
that $f\colon \la\to \vf(M)$ is an action of $\la$ on $M$ such that every
$f(a)$ is a Killing vector field.
Define
\begin{equation}\label{alg-conn-Kill}
\widehat{\chi} = \flat_g \circ \widehat{f} \circ \sharp_{\langle \cdot, \cdot \rangle},
\end{equation}
where $\widehat{f}: M \times {\mathfrak g} \to TM$ is the vector bundle monomorphism induced by the action~$f$.
Then $\chi\colon \la^*\to \Omega^1(M)$ given by $\chi(\alpha)_p =
\widehat{\chi}(p,\alpha)$ is an algebraic connection for
 the operation $i_f\colon \la \to \Der_{-1}(\Omega^*(M))$.
\end{theorem}
\begin{proof}
If $a \in {\mathfrak g}$, $\alpha \in {\mathfrak g}^*$ and $p \in M$ then, using (\ref{alg-conn-Kill}), it follows that
\begin{equation}\label{first-condition}
i_a (\chi(\alpha))(p) = g_p(\widehat{f}(\sharp_{\langle \cdot, \cdot \rangle}(p,
\alpha), \widehat{f}(p, a)) = \langle\sharp_{\langle \cdot, \cdot \rangle}(p,
\alpha), (p,a) \rangle = \alpha(a).
\end{equation}
It is left to check that for any $X\in \vf(M)$, we have
\begin{equation}
\label{asd}
\left( \lie_{f(a)} \chi(\alpha) \right)(X) =  \chi (\ad_a^*\alpha) (X).
\end{equation}
Now, if $b \in {\mathfrak g}$, from (\ref{first-condition}), it follows that
\begin{equation}\label{second-condition-1}
\left({\mathcal L}_{f(a)}\chi(\alpha)\right) (f(b)) = -\chi(\alpha)[f(a), f(b)] = -\chi(\alpha)(f[a, b]) = -\alpha[a, b].
\end{equation}
This proves (\ref{asd}) for $X = f(b)$.


Now, denote the orthogonal complement of $\mathcal{D}_f$ in $TM$ by ${H}$. Let
$Z\in \Gamma({H})$. We are going to show that for all $a\in \la$, the Lie
derivative $\lie_{f(a)} Z$ of $Z$ is a section of ${H}$. For this it is enough to
verify that for all $b\in \la$, we have $g(\lie_{f(a)}Z, f(b)) =0$. Since
$f(a)$ is a Killing vector field, we get
\begin{align*}
0 &= (\lie_{f(a)}g) (Z,f(b)) = - g (\lie_{f(a)} Z, f(b))  - g(Z, [f(a),f(b)] )\\ &=
 - g(\lie_{f(a)}Z, f(b))  - g(Z,f([a,b])) = -g(\lie_{f(a)}Z,f(b)).
\end{align*}
Denote by $P$ the orthogonal projection from $TM$ on $\mathcal{D}_f$.
Let $\bar{P} := \id - P $.
Then for any $X\in \vf(M)$, we have $PX \in \Gamma(\mathcal{D}_f)$ and
$\bar{P}X \in \Gamma({H})$. Thus, for any $a\in \la$
\begin{equation*}
\lie_{f(a)} X = \lie_{f(a)}PX + \lie_{f(a)}\bar{P}X.
\end{equation*}
Since $\lie_{f(a)}PX \in \Gamma(\mathcal{D}_f)$ and $\lie_{f(a)} \bar{P}X\in
\Gamma({H})$, we get
\begin{equation}\label{eq:lie-P}
\lie_{f(a)}\circ P = P \circ \lie_{f(a)},\quad \lie_{f(a)} \circ \bar{P} =
\bar{P}\circ \lie_{f(a)} .
\end{equation}
From the definition of $\chi$, it follows that for any $\beta\in \la^*$, we have
$(\chi(\beta))(X) = \left( \chi(\beta) \right)(PX)$.
Thus by using \eqref{eq:lie-P} we have
\begin{align*}
(\chi(\ad_a^*\alpha))(X) &= (\chi(\ad_a^*\alpha))(PX),\\
(\lie_{f(a)}\chi(\alpha))(X) & = \lie_{f(a)} (\chi(\alpha)(X)) -\chi(\alpha) (\lie_{f(a)} X)\\
& = \lie_{f(a)} (\chi(\alpha)(PX)) -
\chi(\alpha)(P \lie_{f(a)} X)\\
&= \lie_{f(a)} (\chi(\alpha)(PX)) -
\chi(\alpha)( \lie_{f(a)}PX)  = (\lie_{f(a)} \chi(\alpha))(PX).
\end{align*}
This shows that we have to check \eqref{asd}
only for $X\in \Gamma(\mathcal{D}_f)$. Since both sides of \eqref{asd} are
tensorial in $X$, it is enough to check \eqref{asd} for $X$ of the form
$f(b)$, $b\in \la$. This, by (\ref{second-condition-1}), ends the proof.
\end{proof}
Under additional hypotheses on the action of $\la$ on $M$, the formula for the
algebraic connection $\chi$ can be made more explicit.
\begin{corollary}
\label{alg-conn-orthonormal}
Let $(M,g)$ be a Riemannian manifold, and $\la$ a Lie algebra. Suppose
that $f\colon \la\to \vf(M)$ is an action of $\la$ on $M$ such that every
$f(a)$ is a Killing vector field. Suppose that for every pair $a$, $b$ of
elements in $ \la$, the
functions $g\left( f(a), f(b) \right)$ are constant. Then
for any orthonormal basis $\left\{ e_i \right\}$ of $\la$ and the dual basis
$\left\{ e^i \right\}$ of $\la^*$
\begin{equation*}
\chi (e^i) := g(f(e_i), - )
\end{equation*}
gives an algebraic connection for the action $f \colon \la \to \vf(M) $.
\end{corollary}
\subsection{Chevalley model}\label{Chevalley}
If $\la$ is a reductive Lie algebra and the operation $i\colon \la \to
\Der_{-1}(A)$ admits an algebraic connection $\chi$, then Chevalley Fundamental Theorem
\cite[Theorem~I, sec.~9.3]{greub3} provides a model of $(A_\lie,d)$ constructed
from $(A_{i,\lie},d)$, the primitive elements in $H^*(\la)$, and the connection
$\chi$. In this article we will need only the case when $\la$ is abelian.
As the description and derivation of the Chevalley model drastically simplifies in this
situation, we will present only this case.

If $\la$ is abelian, then \eqref{eq:connection} imply that
\begin{equation*}
\lie_a (\chi(\alpha)) = 0,
\end{equation*}
for all $a\in \la$ and $\alpha\in \la^*$. Thus
\begin{equation}
\label{inv}
\chi(\la^*) \subset A_\lie.
\end{equation}
 We
define $\bar\chi\colon \la^*\to A_2$ to be the composition $d \circ \chi$. As
$d$ commutes with $\lie(a)$ for all $a\in \la$, we get that $\bar\chi
(\la^*)\subset A_\lie$. Moreover, for any $a\in \la$ and $\alpha\in \la^*$, we
have
\begin{equation*}
i_a  \bar\chi (\alpha) = i_a  d \chi(\alpha) = \lie_a \chi(\alpha) - d
i_a (\chi(\alpha)) = 0 - d (\alpha (a))  =0.
\end{equation*}
Thus $\bar\chi(\la^*) \subset A_{i,\lie}$ and we can consider $A_{i,\lie}$ as a $V$-CDGA for $V=\la^*$,
where $\la^*$  is seen as a graded vector space concentrated in degree $1$.
Therefore we can construct the CDGA
$\left(A_{i,\lie}\otimes \bw  \la^*,\, d_{\bar\chi}\right)$.
Then Chevalley Fundamental Theorem in this case can be formulated as follows.
\begin{theorem}\label{chevalley}
The application
\begin{equation*}
\begin{aligned}
f  \colon \left(A_{i,\lie}\otimes
\mbox{$\bw $}\,\la^*,d_{\bar\chi}\right) &\to
(A_{\lie},d)\\[2ex]
a \otimes \alpha_1 \wedge \dots \wedge \alpha_k & \mapsto a
\chi(\alpha_1) \dots \chi(\alpha_k).
\end{aligned}
\end{equation*}
is an isomorphism of CDGAs.
\end{theorem}
To prove Theorem~\ref{chevalley}, we first examine the following partial case.

Let $B$ be a CDGA and $i\colon \la\to \Der_{-1}( B)$ an operation on $B$ with an
algebraic connection $\chi\colon \la^*\to B_1$ such that
$\lie_a =0$ for all $a$. Then $B_\lie=B$ and $B_{i,\lie}=B_i$. Thus
Theorem~\ref{chevalley} implies
\begin{theorem}\label{chevalleybis}
The application
\begin{equation*}
\begin{aligned}
f  \colon \left(B_{i}\otimes
\mbox{$\bw $}\,\la^*,d_{\bar\chi}\right) &\to
(B,d)\\[2ex]
b \otimes \alpha_1 \wedge \dots \wedge \alpha_k & \mapsto b
\chi(\alpha_1) \dots \chi(\alpha_k).
\end{aligned}
\end{equation*}
is an isomorphism of CDGAs.
\end{theorem}
Now we show that Theorem~\ref{chevalley} is a corollary of
Theorem~\ref{chevalleybis}.
\begin{proof}[Proof of Theorem~\ref{chevalley} using Theorem~\ref{chevalleybis}]
Take $B=A_{\lie}$. Then the operation $i\colon \la\to \Der_{-1}(A)$ induces an
operation on $B$. To see this, we have only to check that for every $a\in \la$
and every $b\in B$, one gets $i_a(b) \in B$. In other words, we have to show
that $\lie_{a'} i_a (b) =0$ for all $a$, $a'\in \la$ and $b\in B$. We have
\begin{equation*}\label{;lkj;}
\begin{aligned}
\lie_{a'} i_a b = [\lie_{a'}, i_a] b + i_a \lie_{a'} b = i_{[a',a]} b  + 0 = 0
\end{aligned}
\end{equation*}
where we used first that $\lie_{a'}b =0$ as $b\in B = A_{\lie}$ and then that
$[a',a]=0$ as $\la$ is commutative.

Let us denote the resulting operation on $B$ by $i'$.

Now, by \eqref{inv}, we have that the connection $\chi$ can be
corestricted on $B=A_{\lie}$. Let us denote the resulting map $\la^*\to B$ by
$\chi'$. Then it is straightforward that $\chi'$ is an algebraic connection for
$i'$. Thus we can apply Theorem~\ref{chevalleybis} to $B$, $i'$, and $\chi'$.
But now we recovered the map from the claim of Theorem~\ref{chevalley}. This
shows that Theorem~\ref{chevalley} is a consequence of
Theorem~\ref{chevalleybis}.
\end{proof}
In order to prove Theorem~\ref{chevalleybis}, we need the following result.
\begin{proposition}
\label{chevalley_base}
Let $(B,d)$ be a CDGA, $D\in \Der_{-1}(B)$ such that $[D,d]=0$, $D^2=0$, and $\eta \in
B_1$ such that $ D \eta = 1$. Denote by $B_D\subset B$ the kernel of $D$.
Then the map
\begin{equation*}
\begin{aligned}
f  \colon \left(B_{D}\otimes
\mbox{$\bw $}\, \left\langle y \right\rangle,\, dy = d\eta \right) &\to
(B,d)\\[2ex]
a \otimes 1 + b \otimes y & \mapsto a+ b\eta
\end{aligned}
\end{equation*}
is an isomorphism of CDGAs.
\end{proposition}
\begin{proof}
Let $b\in B_k$. Then
\begin{equation*}
D(\eta b) = b - \eta Db.
\end{equation*}
Thus every element $b\in B_k$ can be written as $b= D(\eta b) + \eta Db$. Note
that $D(\eta b) \in B_D$  and $Db \in B_D$, as $D^2=0$.
Hence
\[
b=f(D(\eta b)\otimes 1 + (-1)^{k-1}Db\otimes y).
\]
This shows that the map $f$ is surjective.

Now we show that $f$ is injective. Suppose $b_1 \in B_{D,k}$ and $b_{2} \in
B_{D,k-1}$ are such that $f(b_1\otimes 1 + b_2\otimes y) = 0$. Then $b_1 + b_2\eta
= 0$.  Applying $D$ and taking into account that $Db_1=Db_2=0$ and $D\eta =1$,
we get that $b_2=0$. But then also $b_1=0$. This proves that $\ker(f) = \left\{
0 \right\}$.

That $f$ is a homomorphism of CDGAs follows from a straightforward
computation.
\end{proof}
\begin{proof}[Proof of Theorem~\ref{chevalleybis}]
Let us choose a basis $a_1$, \dots, $a_n$ of $\la$ and denote by $\alpha_1$,
\dots, $\alpha_n$ the dual basis of $\la^*$. We will assume that $n\ge 2$.
Let
\[
B^{(k)}=\{b\in B \; | \; i_{a_1}b= \dots =i_{a_k}b=0 \}, \, 1\leq k\leq n.
\]
By abuse of notation, $B^{(0)}= B$. As all $i_{a_j}$ are derivations, we get that $B^{(k)}$
is a subalgebra of $B$. Moreover, since $i_{a_j}$ commute with the differential
of $B$, the algebras $B^{(k)}$ are endowed with the induced CDGA-structure.

Let us fix $0\le k\le n-1$. Define the derivation $D$ on $B^{(k)}$ by
\begin{equation*}
D b = i_{a_{k+1}}b.
\end{equation*}
To verify that $D$ is well defined we have to show that $Db \in B^{(k)}$. Let
$j\le k$. Then by Lemma~\ref{anticommuting}
\begin{equation*}
i_{a_j} D b =i_{a_j} i_{a_{k+1}} b =- i_{a_{k+1}} i_{a_{j}}b =0.
\end{equation*}
From the axioms of  operation it follows that $D^2=0$. Moreover, $[D,d]=0$ as
$\lie_{a_{k+1}}= 0$. Denote  $\chi(\alpha_{k+1})$ by $\eta$. From
the definition of an algebraic connection it follows that $D\eta =
i_{a_{k+1}} \chi(\alpha_{k+1}) = \alpha_{k+1}(a_{k+1}) =1 $
and for $j\le k$
\begin{equation*}
i_{a_j} \eta = i_{a_j} \chi(\alpha_{k+1}) = \alpha_{k+1} (a_j) =0.
\end{equation*}
Hence $\eta \in B^{(k)}$. Thus we can apply
Proposition~\ref{chevalley_base} to $B^{(k)}$ with the above defined $D$ and
$\eta$.
Note that $B^{(k)}_{i(a_{k+1})}$ coincides with $B^{(k+1)}$.
We get the isomorphism of CDGAs
\begin{equation*}
\begin{aligned}
f_k \colon (B^{(k+1)} \otimes  \mbox{$\bw $}\,\left\langle y_{k+1}
\right\rangle, dy_{k+1} = d\chi(\alpha_{k+1})) & \to  ( B^{(k)},d )\\[2ex]
b \otimes 1 & \mapsto b\\
b \otimes y_{k+1} & \mapsto b \chi(\alpha_{k+1}).
\end{aligned}
\end{equation*}
Note, that for every $j\le k$, we have $d \chi(\alpha_j) \in B^{(k+1)} \subset
B^{(k)}$.  To see this, we have to verify that $i_{a_s} d\chi(\alpha_j) =0$ for
all $s\le k+1$. But $i_{a_s} (\chi(\alpha_j)) = \alpha_j(a_s) =\delta_{js}$, since
$\chi$ is an algebraic connection for the operation $i$. Thus, as $[i_a,d]=0$
for all $a\in \la$, we get
\begin{equation*}
i_{a_s} d \chi(\alpha_j) = - d i_{a_s} \chi(\alpha_j) = -d\delta_{js}=0.
\end{equation*}
Hence the maps
\begin{equation*}
\begin{aligned}
h_k \colon \left\langle y_{k}, \dots, y_1 \right\rangle &\to B^{(k+1)} \otimes
\bw \left\langle y_{k+1} \right\rangle \\
y_j &\mapsto d\chi(\alpha_j) \otimes 1\\
h'_k \colon \left\langle y_{k}, \dots, y_1 \right\rangle &\to B^{(k)}
 \\
y_j &\mapsto d\chi(\alpha_j)
\end{aligned}
\end{equation*}
are well defined. Moreover, $dh_k=0$, $dh'_k=0$ and $f_k\circ h_k =h'_k$.
Hence $f_k$ is a homomorphism of $V$-CDGAs for $V=\left\langle y_{k}, \dots, y_1 \right\rangle$.
By Remark~\ref{iso-iso}, we get the isomorphism $\tilde{f}_k$  of CDGAs
from
\begin{equation*}
 \left(B^{(k+1)} \otimes  \bw \,\left\langle y_{k+1},\dots, y_1
\right\rangle, dy_{k+1} = d\chi(\alpha_{k+1}), \dots, dy_1 = d\chi(\alpha_1)\right)
\end{equation*}
to
\begin{equation*}
  \left( B^{(k)}\otimes \bw \left\langle y_{k},\dots, y_1
\right\rangle,dy_{k} = d\chi(\alpha_{k}),\dots, dy_1 = d\chi(\alpha_1)\right) 
\end{equation*}
defined by
\begin{equation*}
\tilde{f}_k (b\otimes 1) =b\otimes 1,\quad \tilde{f}_k (1\otimes y_{k+1}) =
\chi(\alpha_{k+1}) \otimes 1,\quad  \tilde{f}_k (1\otimes y_j) = 1 \otimes y_j ,\quad 
j\leq k.
\end{equation*}
It is not difficult to check that
$f$  equals to $\tilde{f}_0\circ \tilde{f}_1 \circ \dots \circ \tilde{f}_{n-1}$.
Thus $f$ is a quasi-isomorphism of CDGAs. Then the claim follows upon the identification of
$\left\langle y_{n},\dots, y_1 \right\rangle$ with $\la^*=\left\langle \alpha_{n},\dots, \alpha_1 \right\rangle$.
\end{proof}
%
\section{Models of quasi-Sasakian manifolds}\label{quasi-Sasakian}
First of all, we will recall the definition of a quasi-Sasakian structure as a
particular class of an almost contact metric structure (for more details, see
\cite{blair1,blair2}).

An almost contact metric structure on a manifold $M$ of dimension $2n+1$ is
given by  an endomorphism  $\varphi$ of $TM$, a vector field $\xi$, a $1$-form $\eta$ and a Riemannian metric $h$ satisfying the following conditions
\begin{equation*}
\varphi^2 = -\id + \eta \otimes \xi,\quad  \eta(\xi) = 1,\quad
h \circ \left( \varphi\otimes \varphi \right) = h - \eta \otimes \eta.
\end{equation*}
 A manifold $M$ endowed with an almost contact metric structure is said to be an almost contact metric manifold. The vector field $\xi$ is called the Reeb vector field of $M$. Note that
\[
\varphi(\xi) = 0, \; \; \eta(X) = h(X, \xi),
\]
for $X \in {\mathfrak X}(M)$. In particular, we can consider the free action $f: \mathbb{R} \to {\mathfrak X}(M)$ of the abelian Lie algebra $\mathbb{R}$ on $M$ given by
\[
f(a) = a \xi, \; \; \; \mbox{ for } a\in \mathbb{R}.
\]
For an almost contact metric structure $(\varphi, \xi, \eta, h)$ on $M$, the fundamental $2$-form $\Phi$ is defined by
\[
\Phi(X, Y) = h(X, \varphi Y), \; \; \mbox{ for } X, Y \in {\mathfrak X}(M).
\]
The almost contact metric structure $(\varphi, \xi, \eta, h)$ is said to be
\begin{itemize}
\renewcommand\labelitemi{--}
 \item \emph{normal} if $N_{\varphi} +  d\eta \otimes \xi = 0$, where $N_{\varphi}$ is the
Nijenhuis torsion of $\varphi$;\item \emph{co-K\"ahler} if it is normal, $d\eta = 0$ and $d
\Phi =0$;\item  \emph{Sasakian} if it is normal and $d\eta = \Phi$;\item
\emph{quasi-Sasakian} if it
is normal and $d\Phi = 0$.
\end{itemize}
A standard example of a quasi-Sasakian manifold is the nilpotent Lie group
\[
G = \heis (1, l) \times \mathbb{R}^{2(n-l)},
\]
where $\heis(1, l)$ is the generalized Heisenberg group of dimension $2l + 1$. We remind to the reader that the Heisenberg group $H(1,l)$ is the Lie subgroup of
dimension $2l+1$ in the general linear group $\mathrm{GL}_{l+2}(\R)$ with
elements of the form
\begin{equation*}
\left(
\begin{array}{ccc}
1 & P & t \\
0 & I_l & Q \\
0 & 0 & 1
\end{array}
 \right),
\end{equation*}
where $I_l$ denotes the $l\times l$ identity matrix.
We can take a basis of {left}-invariant $1$-forms $\{\alpha_1, \dots, \alpha_{2l+1}, \beta_1, \dots, \beta_{2(n-l)}\}$ on $G$ given by
\begin{equation*}\label{basis-left-forms}
\alpha_i = dp_i, \; \; \alpha_{l+i} = dq^{i}, \; \; \alpha_{2l+1} = dt - \sum_{i=1}^{l}p_i dq^{i}, \; \; \beta_k =  dx_k,
\end{equation*}
for $i \in \{1, \dots, l\}$ and $k \in \{1, \dots, 2(n-l)\}$.
It is clear that
\[
d \alpha_j = 0, \; \; d\alpha_{2l+1} = -\sum_{j =1}^{l} \alpha_j \wedge \alpha_{l+j} \mbox{ and } d\beta_k = 0,
\]
with $j \in \{1, \dots, 2l\}$ and $k \in \{1, \dots, 2(n-l)\}$. Then, if we denote by
\[
\{X_1, \dots, X_{2l+1}, Y_1, \dots, Y_{2(n-l)} \}
\]
the dual basis of vector fields, we have that
\begin{equation}\label{basis-left-vector}
X_{i} = \displaystyle \frac{\partial}{\partial p_i}, \; \; X_{l+i} = \displaystyle \frac{\partial}{\partial q^i} + p_i \frac{\partial}{\partial t}, \; \; X_{2l +1} = \displaystyle \frac{\partial}{\partial t}, \; \; Y_k = \displaystyle \frac{\partial}{\partial x_k},
\end{equation}
for $i \in \{1, \dots, l\}$ and $k \in \{1, \dots, 2(n-l)\}$.
Now, we can define the left-invariant quasi-Sasakian structure $(\varphi, \xi, \eta, h)$ on $G$ given by
\begin{equation*}\label{varphi-quasi}
\varphi = \sum_{i=1}^{l} (\alpha_i \otimes {X_{l+i}} - \alpha_{l+i} \otimes X_{i}) + \sum_{j=1}^{n-l}(\beta_j \otimes Y_{{n-l}+j} - \beta_{{n-l}+j} \otimes Y_j)
\end{equation*}
and
\begin{equation*}\label{xi-eta-quasi}
\xi = X_{2{l}+1}, \; \; \eta = \alpha_{2{l}+1} \; \; \mbox{ and } \; \; h = \sum_{i=1}^{2{l}+1} \alpha_i \otimes \alpha_i + \sum_{k=1}^{2{(n-l)}} \beta_k \otimes \beta_k,
\end{equation*}
with fundamental $2$-form
\begin{equation*}
\Phi = \sum_{i=1}^{l} \alpha_{l+i} \wedge \alpha_i - \sum_{j=1}^{n-l} \beta_j \wedge \beta_{{n-l}+j}.
\end{equation*}
So, if $\Gamma$ is a cocompact discrete subgroup of $G$, then $(\varphi, \xi,
\eta, h)$ induces a quasi-Sasakian structure on the compact nilmanifold $\Gamma \backslash G$.\label{S} In other words, $\Gamma \backslash G$ is a compact quasi-Sasakian nilmanifold.
Note that $G$ is a nilpotent Lie group and the structure constants of its Lie
algebra with respect to the previous basis are rational numbers. Therefore,
$G$ admits a cocompact discrete subgroup (see~\cite{malcev}).

\begin{remark} If ${n = l}$ in the previous example, then the
quasi-Sasakian structure on   $\Gamma \backslash G$  is Sasakian and if $ {l} =0$ then it is
co-K\"ahler. However, if $ {n \neq l}$ and $ {l \neq 0}$ then the compact
nilmanifold  {$\Gamma \backslash G$} does not admit either a Sasakian or a co-K\"ahler
structure. In
fact, a compact co-K\"ahler nilmanifold is diffeomorphic to a torus and a
compact Sasakian nilmanifold is diffeomorphic to a compact quotient of a
Heisenberg group of odd dimension with a cocompact discrete subgroup (see
\cite{nilsasakian}). So, we can conclude that the class of the compact
quasi-Sasakian manifolds is actually distinct from the classes of compact
co-K\"ahler and compact Sasakian manifolds.
\end{remark}

%
Now, we will show that on a quasi-Sasakian manifold the foliation of rank $1$ generated by the Reeb vector field is transversely K\"ahler.
K\"ahler manifolds are defined as a special case of Hermitian manifolds.

An almost Hermitian structure on a manifold $M$ of even dimension $2(n+1)$ is a couple $(J, g)$, where $J$ is a $(1, 1)$ tensor field on $M$, $g$ is a Riemannian metric and
\begin{equation*}
J^2 = -\id, \; \; g(JX, JY) = g(X, Y),
\end{equation*}
for $X, Y \in {\mathfrak X}(M)$.

The fundamental $2$-form of $M$ is defined by
\[
\Omega (X, Y) = g(X, JY), \; \; \mbox{ for } X, Y \in {\mathfrak X}(M)
\]
A manifold $M$ endowed with an almost Hermitian structure is said to be an
\emph{almost Hermitian manifold}.
The almost Hermitian manifold $(M, J, g)$ is said to be:
\begin{itemize}
\renewcommand\labelitemi{--}
\item Hermitian if $N_J = 0$, where $N_J$ is the Nijenhuis torsion of $J$;\item K\"ahler
if it is Hermitian and $d\Omega = 0$.
\end{itemize}
\begin{definition}
\label{leeoneform}
The $1$-form $\theta:= \frac1{n}\delta \Omega \circ J$ on a Hermitian
manifold $\left( M,J,g \right)$ is called the \emph{Lee $1$-form}, where $\delta$ is the codifferential.
\end{definition}

Let $\fol$ be a foliation  on a manifold $M$ of codimension $q$. Denote by $\nu \fol$
the vector bundle $TM/T\fol$ on $M$. Given a foliated chart $U\subset M$  for $\fol$, we
have the quotient map $f_U \colon U \to \R^q$.
Note, that $f_U$ induces an isomorphism $f_{U,p} \colon (\nu \fol)_p \to T_{f_U(p)} \R^q$ for
every point $p\in U$.

If $U$ and $V$ are two foliated charts with a non-empty intersection then there is a
smooth function $\tau_{UV} \colon f_U(U\cap V) \to f_V(U\cap V)$ such that $f_V =
\tau_{UV}\circ f_U$ on $U\cap V$.

The foliation $\fol$ is called \emph{transversely K\"ahler}, if for every foliated chart
$U$ there is given a K\"ahler structure on $f(U)$ so that every transition
function $\tau_{UV}$ preserves the K\"ahler structure.  

An endomorphism $J$ of the distribution associated to $\fol$ such that $\left[ J,J \right]_{FN} =0$,
$J^2=-\id$, and $\lie_X J=0$ for all $X\in \Gamma \fol$ is called \emph{foliated complex
structure} on $\fol$.
\begin{proposition}
Let $J$ be a foliated complex structure on $\fol$ and $g$ an $\fol$-invariant metric on
$\fol$. Define $\Omega (X,Y) = g(X,JY)$ for $X$, $Y\in \Gamma(\nu\fol)$.
If $d\Omega =0$, then $\fol$ is transversely K\"ahler.
\end{proposition}
\begin{proof}
Let $U$ be a foliated chart and $f=f_U\colon U\to \R^q$ the corresponding projection.
Since $\lie_X J=0$ and $\lie_X g=0$ for all $X\in \Gamma(\fol) $, we have
a well-defined almost complex structure $J'$ on $f(U)$ and a well-defined
Riemannian metric $g'$ on $f(U)$ induced by $J$ and $g$, respectively. More
precisely, given
 a point $x\in f(U)$ choose an arbitrary  $p\in f^{-1}(x)$.
We have the isomorphism
$h:=f_{U,p}\colon (\nu\fol)_p \to T_x \R^q$
Then
\begin{equation}
\label{primes}
J' X = h J h^{-1}(X),\quad  g'(X,Y) = g(h^{-1} X,
h^{-1}Y)
\end{equation}
for any $X$, $Y\in T_x f(U)$.
Since $J$ is integrable and $d\Omega =0$,  standard computations show that
$(J',g')$ is a K\"ahler structure on $f(U)$.

Now, let $V$ be another foliated chart and denote by $(J'',g'')$ the
corresponding K\"ahler structure on $f(V)$.
Then using~\eqref{primes} and similar formulas for $J''$ and $g''$, it is easy to see that $\tau_{UV}$ is a holomorphic
isometry. This shows that $\fol$ is transversely K\"ahler.
\end{proof}
\begin{proposition}\label{q-s-trans-Kahler}
On a quasi-Sasakian manifold the foliation of rank $1$ generated by the Reeb vector field is transversely K\"ahler.
\end{proposition}
\begin{proof}
If $(\varphi, \xi, \eta, h)$ is the almost contact metric structure on the quasi-Sasakian manifold $M$ then
${\mathcal L}_{\xi}\varphi = 0$
(see, for instance, Theorem 6.1 in \cite{blair2}). Thus, using that $\xi$ is a Killing vector field, we have that the couple $(\varphi, h)$ induces a transverse K\"ahler structure $(J, g)$ with respect to the foliation of rank $1$ generated by $\xi$. In fact, since $N_{\varphi} +  d \eta \otimes \xi = 0$, it follows that the transverse Nijenhuis torsion of $J$ is zero. In addition, the transverse fundamental $2$-form of $(J, g)$ is just the fundamental $2$-form $\Phi$ of $M$ which is basic and closed.
\end{proof}
Now we are ready to describe a model for quasi-Sasakian manifolds that
generalizes the Tievsky model~\cite{tievsky} for the Sasakian manifolds.
\begin{theorem}\label{model_qS}
Let $(M^{2n+1},  \varphi, \xi, \eta, g)$ be a compact quasi-Sasakian manifold.
Then the CDGA
\begin{equation}
\label{eq:model_qS}
\left(H^*_B(M,\xi) \otimes \bw \left\langle y \right\rangle, dy =
[d\eta]_B\right)
\end{equation}
is quasi-isomorphic to $\Omega^*(M)$. In
other words the CDGA~\eqref{eq:model_qS}
is a model of~$M$.
\end{theorem}
\begin{proof}
Since $\xi$ is a Killing vector field,
 it follows from Theorem~\ref{invariant}
that the inclusion $\Omega_{\lie_\xi}^*\to \Omega^*(M)$ is a quasi-isomorphism.
As $\xi$ is Killing and of constant length,
by Corollary~\ref{alg-conn-orthonormal}
the map $\chi \colon \R \to \Omega^1(M) $ given by $\chi(t) = t\eta$
is an algebraic connection for
the operation $i$ on $\Omega^*(M)$. Therefore by Theorem~\ref{chevalley} the
CDGA
\begin{equation*}
\left( \Omega^*_B(M,\xi) \otimes \bw \left\langle y \right\rangle, dy=d\eta
\right)
\end{equation*}
is quasi-isomorphic to $\Omega_{\lie_\xi}^*(M)$ and thus to $\Omega^*(M)$.

By Proposition ~\ref{q-s-trans-Kahler} the foliated manifold $(M,\xi)$ is transversely K\"ahler.
Now,  by \cite[Theorem 3]{wolak}  the CDGA
$\Omega^*_B(M,\xi)\otimes \C$
over $\C$ is formal. It is proved in \cite[Theorem~12.1]{sullivan} that the
property to be formal or non-formal is preserved under field extensions. Thus
$\Omega^*_B(M,\xi)$ is  a formal CDGA over $\R$. In other words, $\Omega^*_B(M,\xi)$ is
quasi-isomorphic to $(H^*_B(M,\xi),0)$. By Corollary~\ref{substitution}, we get
that
\begin{equation*}
\left(\Omega^*_B(M,\xi)\otimes \bw \left\langle y \right\rangle,
dy=d\eta\right)
\end{equation*}
and
\begin{equation*}
\left(H^*_B(M,\xi)\otimes \bw \left\langle y \right\rangle, dy
=[d\eta]_B\right)
\end{equation*}
are quasi-isomorphic. This proves the theorem.
\end{proof}
Motivated by the models described in Theorem~\ref{model_qS}, we introduce the
following class of CDGAs.
\begin{definition}\label{almostformal}
We say that a CDGA $(B,d)$ is \emph{almost formal} of  \emph{index}
$l$ if it is quasi-isomorphic to
the CDGA $\left(A\otimes \bw \left\langle y \right\rangle,dy=z\right)$, where $A$ is a
connected CDGA with the zero differential and $z\in A_2$ is a
closed homogeneous element satisfying  $z^l\not=0$, $z^{l+1}=0$.
\end{definition}

The previous definition and Theorem \ref{model_qS} suggest us to introduce the following notion for quasi-Sasakian manifolds.
\begin{definition}
Let $(M^{2n+1}, \varphi, \xi, \eta, h)$ be a quasi-Sasakian manifold. The
\emph{index} of $M$ is the natural number $l$, $0 \leq l \leq n$, satisfying
\[
[d \eta]^l_B \neq 0 \; \; \mbox{ and } \; \; [d \eta]^{l+1}_B = 0.
\]
\end{definition}
\begin{remark}\label{index-Sasakian}
If $M$ is a compact Sasakian (resp. co-K\"ahler) manifold of dimension $2n+1$
then, from Lemma 3.1 in \cite{nilsasakian}, we have that the index of $M$ is
maximal and equal to $n$ (resp., minimal and equal to $0$).
\end{remark}
Using Theorem \ref{model_qS}, we deduce that the model ~\eqref{eq:model_qS} of a
compact quasi-Sasakian manifold of index $l$ is an almost formal CDGA of the
same index.

\section{Models of quasi-Vaisman manifolds}
\label{quasi-Vaisman}
In this section, we will introduce a particular class of Hermitian structures as a natural extension of Vaisman structures.

Recall that a Hermitian manifold $(M, J, g)$ is said to be
\emph{locally conformal K\"ahler} (or LCK) if the fundamental 2-form $\Omega$
and the Lee  $1$-form $\theta$ satisfy the identities
\[
d \Omega = \Omega \wedge \theta,\quad d\theta=0.
\]
 The manifold is said to be a \emph{Vaisman manifold}  if, moreover, the Lee 1-form is parallel with respect to the Levi-Civita connection of $g$.
The \emph{anti-Lee $1$-form}  $\eta$ is defined as
$\eta = -\theta \circ J$ while the \emph{Lee} and \emph{anti-Lee vector fields} $U,V$ are defined as the metric duals of $\theta, \eta$, respectively.

%

Let $(M, J, g)$ be a Vaisman manifold with Lee and anti-Lee $1$-forms $\theta$ and $\eta$, respectively, and Lee and anti-Lee vector fields $U$ and $V$, respectively.
We will assume (without loss of generality) that the norm of $\theta$ is~$1$. In
these conditions, one can prove that
\[
d \Omega = d\eta \wedge \theta,
\]
and that, moreover, the Lee vector field $U$ is Killing and an infinitesimal
automorphism of the complex structure, that is,
\[
{\mathcal L}_U g = 0 \; \; \mbox{ and } \;  \; {\mathcal L}_U J = 0
\]
(see, for instance, Propositions 4.2 and 4.3 in \cite{DrOr}; see also \cite{vaisman_roma}).

Motivated by the previous results, we introduce the following definition.
\begin{definition}
A \emph{quasi-Vaisman structure} on a manifold $M$ is a triple $(J, g, \theta)$, with $(J, g)$ a Hermitian structure, $\theta$ a closed $1$-form and such that the metric dual $U$ of $\theta$ is unitary,
Killing and, in addition,
\[
\lie_U J =0,\quad d \Omega = d\eta \wedge \theta,
\]
where $\Omega$ is the fundamental $2$-form of $M$ and
\begin{equation}\label{eta-theta}
\eta = -\theta \circ J.
\end{equation}
\end{definition}
On a quasi-Vaisman manifold, with quasi-Vaisman structure $(J, g, \theta)$, the closed $1$-form $\theta$ does not coincide, in general, with the Lee $1$-form of the Hermitian manifold $(M, J, g)$. For this reason, we will use the following terminology. The $1$-form $\theta$ is called  the \emph{quasi-Lee $1$-form} and its metric dual $U \in {\mathfrak X}(M)$ is called the \emph{quasi-Lee vector field}. The $1$-form $\eta = - \theta \circ J$ is the \emph{quasi-anti-Lee $1$-form} and its metric dual $V \in {\mathfrak X}(M)$ is the \emph{quasi-anti-Lee vector field}. It is clear that
\begin{equation}\label{U-V}
JU = V, \; \; \; JV = -U.
\end{equation}
\begin{remark}
Let $(M, J, g)$ be a quasi-Vaisman manifold. Then, using the general relation
\[
2 g(\nabla_X U, Y) = ({\mathcal L}_U g)(X, Y) + d\theta(X, Y), \; \; \mbox{ for } X, Y \in {\mathfrak X}(M),
\]
where $\nabla$ is the Levi-Civita connection of $g$, we deduce that $U$ (resp. $\theta$) is a parallel vector field (resp. $1$-form).
\end{remark}
\begin{remark}
It is easy to check that a quasi-Vaisman manifold $\left( M,J,g,\theta \right)$ is Vaisman if and only
if $\Omega = d\eta + \theta \wedge \eta$, i.e. if and only if it is an l.c.s.
manifold
of first kind.
\end{remark}
Quasi-Sasakian and quasi-Vaisman manifolds are closely related. We will show in
Section~\ref{mappingtorusquasivaisman} that if $(N, \varphi, \xi, \eta, h)$ is a quasi-Sasakian manifold then the product of $N$ with the real line $\mathbb{R}$ or the circle $S^1$ admits a quasi-Vaisman structure $(J, g)$, with $J$ and $g$ given by
\[
J = \varphi + \xi \otimes \theta - E \otimes \eta, \; \; g = h + \theta \otimes \theta,
\]
where $\theta$ is the standard volume form on $\mathbb{R}$ or on $S^1$ and $E$
is the dual vector field to $\theta$. In particular, the nilpotent Lie group
\[
G = {\heis(1, l) \times \mathbb{R}^{2(n-l)+1}}
\]
admits a {left}-invariant quasi-Vaisman structure. Thus, if $\Gamma$ is a cocompact discrete subgroup then the compact nilmanifold
${\Gamma \backslash G}$ admits a quasi-Vaisman structure.\label{V}
\begin{remark} Note that if ${n = l}$ in the previous example, then the
quasi-Vaisman structure on ${\Gamma \backslash G}$ is Vaisman and if ${l =0}$ then it is
K\"ahler. However, if {$n \neq l$ and $ l \neq 0$} then the compact
nilmanifold {$\Gamma \backslash G$} doesn't admit either  a Vaisman or  a K\"ahler
structure. Indeed, a compact K\"ahler nilmanifold is diffeomorphic to a torus
(see \cite{BeGo,hasegawa}) and a compact Vaisman nilmanifold is diffeomorphic to a compact quotient of a product $\heis(1, k) \times \mathbb{R}$ by a cocompact discrete subgroup (see \cite{bazzoni}). So, we can conclude that the class of the compact quasi-Vaisman manifolds is distinct from  the classes of compact Vaisman and K\"ahler manifolds.
\end{remark}

Next, we will see that the {quasi} Lee and anti-Lee vector fields in a quasi-Vaisman manifold $M$ induce a free action of the abelian Lie algebra $\mathbb{R}^2$ on $M$.
\begin{proposition}
\label{qvaismanaction}
Let $(M, J, g)$ be a quasi-Vaisman manifold and $U, V$ the {quasi} Lee and anti-Lee vector fields on $M$. Then, the map $f: \mathbb{R}^2 \to {\mathfrak X}(M)$ given by
\[
f(a, b) = a U + b V, \; \; \; \mbox{ for } a, b \in \mathbb{R}
\]
is a free action of the abelian Lie algebra $\mathbb{R}^2$ on $M$.
\end{proposition}
\begin{proof}
Using {(\ref{U-V})} and the fact that $U$ is unitary, we deduce that $V$ is
also unitary. Thus, from {(\ref{eta-theta})}, it follows that
\begin{equation}\label{normalizar}
\theta(U) = \eta(V) = 1, \; \; \; \theta(V) = \eta(U) = 0,
\end{equation}
which implies that the vector fields $U$ and $V$ generate a distribution of rank $2$ on~$M$.

Next, we show that
\[
[U, V] = 0.
\]
In fact, a direct computation proves that
\[
({\mathcal L}_U \Omega)(X, JY) = -({\mathcal L}_Ug)(X, Y) + g(X, ({\mathcal L}_UJ)(JY)),
\]
and, therefore,
\begin{equation}\label{U-automorphism}
{\mathcal L}_U\Omega = 0.
\end{equation}
On the other hand, using {(\ref{U-V})}, it follows that
\begin{equation}\label{V-dual-form}
i_V\Omega = \theta.
\end{equation}
Then, from (\ref{U-automorphism}) and (\ref{V-dual-form}), we have that
\[
i_{[U, V]} \Omega = i_V{\mathcal L}_U \Omega - {\mathcal L}_Ui_V\Omega = - {\mathcal L}_U\theta
\]
and, using (\ref{normalizar}) and the fact that $\theta$ is closed, we deduce that
\[
i_{[U, V]}\Omega = 0.
\]
Since $\Omega$ is non-degenerate, this implies that $[U, V] = 0$.
\end{proof}
Next, we will prove that the {quasi} anti-Lee vector field of a quasi-Vaisman manifold
is  Killing and an infinitesimal automorphism of the complex structure.
We will need the
following result.
\begin{lemma}\label{u-v-invariant}
On a quasi-Vaisman manifold $M$ we have that
\[
{\mathcal L}_{U} \eta = 0,\quad {\mathcal L}_V \theta = 0,\quad {\mathcal L}_V \eta = 0,
\]
where $U, V$ are the {quasi}-Lee and anti-Lee vector fields,  and $\theta, \eta$ are the {quasi}-Lee and anti-Lee $1$-forms.
\end{lemma}
\begin{proof} First we will show that
\begin{equation}\label{v-U-invariant}
{\mathcal L}_U \eta = 0.
\end{equation}
If $X$ is a vector field on $M$ then, using {(\ref{eta-theta})}, it follows that
\[
({\mathcal L}_U \eta)(JX) = U(\theta(X)) + \theta(J[U, JX]).
\]
Now, since $\lie_U J=0$, we have  $[U, JX] = J[U, X]$ and we deduce that
\[
({\mathcal L}_U \eta)(JX) = ({\mathcal L}_U\theta)(X) = d\theta(U, X) + d(\theta(U))(X) = 0.
\]
On the other, using {again} that $\theta$ is closed and (\ref{normalizar}), we obtain that
\[
{\mathcal L}_V \theta = i_V(d\theta) + d(\theta(V)) = 0.
\]
Finally, we will prove that ${\mathcal L}_V \eta = 0$. In fact, if $X$ is a vector field on $M$ then, from {(\ref{eta-theta}), (\ref{U-V})} and since $N_J(U, X) = 0$, we have that
\[
({\mathcal L}_V \eta)(JX) = V(\theta(X)) + \theta([JU, X] + [U, JX] -J[U, X]) = ({\mathcal L}_V \theta)(X) + \theta(({\mathcal L}_U J)(X)).
\]
Thus, using (\ref{v-U-invariant}) and the fact that $\lie_U J=0$, we conclude that
\mbox{$ ({\mathcal L}_V \eta)(JX) = 0$}.
\end{proof}
Next, using the previous result, we will prove that the flat foliation generated by $U$ and $V$ is transversely K\"ahler.
\begin{proposition}\label{quasi-Vais-trans-Ka}
On a quasi-Vaisman manifold the flat foliation generated by the {quasi} Lee and anti-Lee vector fields is transversely K\"ahler.
\end{proposition}
\begin{proof}
Let $(M, J, g)$ be a quasi-Vaisman manifold with {quasi} Lee and anti-Lee vector fields $U$ and $V$, respectively. Then, we have that
\[
{\mathcal L}_U J = 0, \; \; {\mathcal L}_Ug = 0.
\]
Next, we will show that $\lie_V J=0$. Indeed, if $X$ is a vector field on $M$
then, since $N_J(U, X) = 0$ and $JU=V$, we deduce that
\[
({\mathcal L}_V J)(X) = J[U, JX] + [U, X] = -({\mathcal L}_UJ)(JX),
\]
which, using $\lie_U J=0$, implies that $({\mathcal L}_V J)(X) = 0$.
Now, we will prove that ${\mathcal L}_V\Omega = 0$. In fact,
\[
{\mathcal L}_V\Omega = d(i_V\Omega) + i_V(d\Omega) = d\theta + i_V(d\eta \wedge \theta) = i_V(d\eta \wedge \theta).
\]
Therefore, from (\ref{normalizar}) and Lemma \ref{u-v-invariant}, we obtain that
\[
{\mathcal L}_V\Omega = {\mathcal L}_V \eta \wedge \theta = 0.
\]
Now, we will see that $V$ is Killing. If $X, Y$ are vector fields on $M$, we deduce that
\[
0 = ({\mathcal L}_V\Omega)(X, JY) = - ({\mathcal L}_Vg)(X, Y) + g(X, ({\mathcal L}_VJ)(JY))
\]
and, since $\lie_V J=0$, we conclude that ${\mathcal L}_V g = 0$.

Thus, the Hermitian structure $(J, g)$ induces a transversely K\"ahler structure {$(\hat{J}, \hat{g})$} on $M$ with respect to the flat foliation {$\fol$} generated by $U$ and $V$. {In fact, using that $N_J = 0$, we deduce that the transverse Nijenhuis torsion of $\hat{J}$ is zero. On the other hand, the transverse K\"ahler $2$-form $\hat{\Omega} = \Omega - \eta \wedge \theta$ is basic and closed.}
\end{proof}
Now we use Theorem~\ref{chevalley} to provide a  model for quasi-Vaisman
manifolds.
\begin{theorem}\label{model_pV}
Let $(M^{2n+2},J,  g)$ be a compact quasi-Vaisman manifold and $U$, $V$, $\fol$
defined as above. Then the CDGA
\begin{equation}
\label{eq:model_pV}
\left(H^*_B(M,\fol) \otimes \bw \left\langle x,y \right\rangle, dx = 0, dy=
[d\eta]_B \right)
\end{equation}
is quasi-isomorphic to $\Omega^*(M)$. In
other words the CDGA~\eqref{eq:model_pV} is a model of~$M$.
\end{theorem}
\begin{proof}
We consider the action $f\colon \R^2\to \vf(M)$ defined in
Proposition~\ref{qvaismanaction}.
Since the image of $f$ is generated by $U$ and $V$ and they are Killing, we get by Theorem~\ref{invariant} that
the inclusion $\Omega^*_{\lie_U,\lie_V}(M) = \Omega^*_{\lie_f}(M)  \to \Omega^*(M)$ is a
quasi-isomorphism.

Since the vector fields $U$ and $V$ are unitary and mutually orthogonal, we can apply
Corollary~\ref{alg-conn-orthonormal} to compute an algebraic connection
$\chi$ for the operation $i_f$.
For an appropriate choice of a basis $\left\{ x,y \right\}$ of $\left( \R^2
\right)^*$ we get
\begin{equation*}
\chi(x) = g(U,-) = \theta, \quad \chi(y) = g(V,-) = \eta.
\end{equation*}
By Theorem~\ref{chevalley}
we obtain that
$$\left(\Omega^*_B(M,\fol)\otimes \bw \left\langle x,y \right\rangle, dx =0, dy
=d\eta\right) $$
is quasi-isomorphic to $\Omega^*_{\lie_U,\lie_V}(M)$ and thus also to
$\Omega^*(M)$.

{By Proposition~\ref{quasi-Vais-trans-Ka}} the foliated manifold $(M,\fol)$ is transversely
K\"ahler. Now
one can proceed as in the proof of Theorem~\ref{model_qS}, replacing where
needed $\xi$ with $\fol$,
in order to get
that
\begin{equation*}
\left(\Omega^*_B(M,\fol)\otimes \bw \left\langle x,y \right\rangle,
dx=0,dy=d\eta \right)
\end{equation*}
and
\begin{equation*}
\left(H^*_B(M,\fol)\otimes \bw \left\langle x,y \right\rangle, dx
=0, dy=[d\eta]_B \right)
\end{equation*}
are quasi-isomorphic. This completes the proof of the theorem.
\end{proof}
Note that the model in Theorem~\ref{model_pV} is in fact an almost formal CDGA.
To see this we can take
\begin{equation*}
A:= H_B^*(M, \fol) \otimes \bw \left\langle x \right\rangle
\end{equation*}
and $z=[d\eta]_B$ considered as an element in $A$.

{Now, as in the quasi-Sasakian case, we can also introduce the following definition.}
\begin{definition}
{The index of a quasi-Vaisman manifold $(M^{2n+2}, J, g)$ with quasi anti-Lee $1$-form $\eta$ is the natural number $l$, $0 \leq l \leq n$, which satisfies
\[
[d \eta]_B^{l} \neq 0 \; \; \mbox{ and } \; \; [d\eta]_B^{l +1} = 0.
\]
}
\end{definition}
{So, the model (\ref{eq:model_pV}) of a compact quasi-Vaisman manifold of  index
$l$ is an almost formal CDGA of the same index.}
\begin{remark}\label{index-Vaisman}
{
Let $M$ be a compact Vaisman manifold of dimension $2n+2$. Then, the index of
$M$ is maximal and equal to $n$. Indeed, if $\Omega$  is the fundamental $2$-form of $M$ and $\theta$, $\eta$ are the Lee and anti-Lee $1$-form then, using Proposition 4.3 in \cite{DrOr} (see also \cite{vaisman_roma}), we have that $\Omega = d\eta + \eta \wedge \theta$. On the other hand, if $U$ and $V$ are the Lee and anti-Lee vector fields then, as we know, $i_Ud\eta = i_Vd\eta = 0$ and, since $\Omega$ is non-degenerate, we conclude that
\[
\nu = \theta \wedge \eta \wedge (d\eta)^n
\]
is a volume form on $M$. Now, suppose that the index of $M$ is less than $n$.
Then, there exists a basic $(2n-1)$-form $\mu$ such that  $d\mu = (d\eta)^n$. So,
\[
\nu = \theta \wedge \eta \wedge d\mu = d(\theta \wedge \eta \wedge \mu) + \theta \wedge d\eta \wedge \mu.
\]
But, since $\mu$ is a basic form, we have that $i_V(\theta \wedge d \eta \wedge \mu) = 0$ and, therefore, $\theta \wedge d \eta \wedge \mu = 0$. This implies that $\nu = d(\theta \wedge \eta \wedge \mu)$ which, using that $M$ is compact and that $\nu$ is a volume form, is a contradiction.
}
\end{remark}
\section{Almost formal nilmanifolds}\label{nilmanifolds}


A compact homogeneous space of a nilpotent Lie group is called  a \emph{nilmanifold}. It
was proved by Malcev~\cite{malcev} that every nilmanifold is diffeomorphic to
$\Gamma\backslash G$ for some nilpotent Lie group $G$ and a cocompact subgroup $\Gamma$
of $G$.

In~\cite{hasegawa} Hasegawa determined the minimal Sullivan model of a
nilmanifold using Nomizu theorem. Namely
\begin{theorem}[{\cite{hasegawa}}]\label{hasegawa}
Let $M \cong \Gamma\backslash G$ be a compact nilmanifold. Denote by $\la$ the Lie
algebra of $G$ and by $\left( \bw \la^*,d^{CE} \right)$ its Chevalley-Eilenberg
complex considered as a CDGA with the multiplication of the exterior algebra.
Then $\left( \bw\la^*,d^{CE} \right)$ is a minimal model of $\Omega^*(M)$.
\end{theorem}

The Heisenberg Lie algebra $\heisla(1,l)$ of the Heisenberg Lie group $\heis\left(
1,l \right)$  has the following multiplicative structure with
respect to a suitable basis $p_1$, $p_2$,\dots, $p_l$, $q_1$, $q_2$, \dots, $q_l$,~$h$:
\begin{equation}
\label{mult}
[p_i,p_j] =0, \quad [q_i,q_j]  = 0,\quad [p_i, q_j] = \delta_{ij}h, \quad [p_i,h]=0, \quad
[q_i, h] =0
\end{equation}
for all possible pairs $i$ and $j$. 
Such a basis may be chosen as follows
\[
p_i = X_i, \; \; q_i = X_{l+i}, \; \; h = X_{2l +1}, \; \; \mbox{ for } i \in \{1, \dots, l\},
\]
where $X_i, X_{l+i}$ and $X_{2l+1}$ are the left-invariant vector fields on $H(1, l)$ given by (\ref{basis-left-vector}).  
It is clear that the Lie algebra of
$\heis(1,l)\times \R^r$ is $\heisla(1,l)\oplus \abla_r$, where $\abla_r$
denotes the $r$-dimensional abelian Lie algebra. We will write
$\la$ instead of $\heisla(1,l)\oplus \abla_r$ in this discussion to avoid cumbersome
formulas.  Choose a basis
$u_1$, \dots,~$u_r$ of $\abla_r$. Then all the elements $u_1$,~\dots,~$u_r$ are
in the center of $\la$ and this with~\eqref{mult}
determines its
multiplicative structure.
Thus the Chevalley-Eilenberg differential on \mbox{$(\heisla(1,l)\oplus
\abla_r)^*$} is
given by
\begin{equation*}
dh^* = - \sum_{j=1}^l p^*_j \wedge q^*_j,\quad dp_j^*=dq_j^* = du_j^* =0,
\end{equation*}
with $\{p_1^*, \dots, p_l^*, q_1^*, \dots, q_l^*, h^*\}$ the dual basis of $\{p_1, \dots, p_l, q_1, \dots, q_l, h\}$. 
Denote by $A$ the  subalgebra of $\bw \la^* $ generated by $p_i^*$, $q_i^*$, and
$u_j^*$. Then
\begin{equation*}
A=\bw \left\langle p_1^*,\dots,p_l^*, q_1^*,\dots, q_l^*, u_1^*,\dots, u_r^* \right\rangle
\end{equation*}
and the restriction of $d^{CE}$ to $A$ is zero.
Thus $\left(\bw \la^*, d^{CE}\right)$ can be identified with $\left(A\otimes \bw\left\langle
h^* \right\rangle, dh^*=z \right)$, where
\begin{equation*}
z=- \sum_{j=1}^l p^*_j \wedge q^*_j.
\end{equation*}
Thus we see that a nilmanifold $M$ modelled on $\heis(1,l)\times \R^r$ has an
almost formal model of index $l$.
In Theorem~\ref{heisenberg} we will show that
these examples exhaust all nilmanifolds having almost formal models.
We will use it in order to classify nilmanifolds admitting quasi-Sasakian or
quasi-Vaisman structure.

We start by  proving a vanishing property for  general almost formal manifolds.
\begin{proposition}
\label{dimension}
Suppose $M$ is an $m$-dimensional manifold that admits an almost formal model
$\left( A\otimes \bw \left\langle y \right\rangle, dy=z \right)$ of index
$l$. Then $A_n=0$ for all
$n\ge m$.
\end{proposition}
\begin{proof}
First we show that for every $n\ge m$ the map
\begin{equation}
\label{anan}
\begin{aligned}
A_n & \to A_{n+2}\\
a & \mapsto az
\end{aligned}
\end{equation}
is injective.
Suppose $a\in A_n\setminus\{0\} $ is such that $az=0$. Then $d(a\otimes y)
=az=0$. Since the image of $d \colon \left( A\otimes \bw\left\langle y
\right\rangle \right)_n \to \left( A\otimes \bw\left\langle y
\right\rangle \right)_{n+1}$ lies inside of $A_{n+1}$, this implies that
$\left[ a \otimes y \right]$ is a non-zero cohomology class in $H^{n+1}\left( A\otimes
\bw\left\langle y \right\rangle
\right) \cong H^{n+1}(M)$. As $M$ is of dimension $m<n+1$, we get a
contradiction. This shows that the maps~\eqref{anan} are injective.
But then also the maps
\begin{equation*}
\begin{aligned}
A_n &\to A_{n+2l+2}\\
a & \mapsto az^{l+1}
\end{aligned}
\end{equation*}
are injective for all $n\ge m$. Since $z^{l+1}=0$ this implies $A_n=0$.
\end{proof}
Now we give a characterization of almost formal nilmanifolds.
\begin{theorem}\label{heisenberg}
Let $G$ be an $m$-dimensional $1$-connected nilpotent Lie group with the Lie algebra
$\la$ and $\Gamma$ a cocompact subgroup of $G$. Then the manifold $\Gamma\backslash G$
admits an almost formal  model of  index $l$ if and only if $G$ is isomorphic to
$\heis(1,l)\times \R^{m-2l-1}$.
\end{theorem}
\begin{proof}
We denote the dimension of $H^1(\Gamma\backslash G) \cong H^1(\bw \la^*)$
by $b_1$. We can choose a basis $\alpha_1$, \dots, $\alpha_m$ of $\la^*$ such
that $\alpha_1$, \dots, $\alpha_{b_1}$ is a basis of $\ker(d_1^{CE})$ and
\begin{equation}
\label{dec}
d_1^{CE}\alpha_k =  \sum_{i<j<k} \gamma^{ij}_k \alpha_i \wedge \alpha_{j}.
\end{equation}
Notice that since $d_0^{CE}=0$, we have $\ker(d_1^{CE}) = H^1(\bw \la^*)$.

Let $\left( A\otimes \bw \left\langle y \right\rangle, dy = z\right)$ be the almost formal
model of $\Gamma\backslash G$ of dimension $m$ and index $l$.
By Theorem~\ref{hasegawa} the CDGA  $\left(\bw \la^*, d^{CE}\right)$ is the minimal model of $\Gamma\backslash G$.
Therefore there is a quasi-isomorphism of CDGA's
\begin{equation*}
\psi \colon \bw \la^* \to A\otimes \bw \left\langle y \right\rangle.
\end{equation*}
By Proposition~\ref{dimension}, we have
$A_m=0$, and hence  $\left( A\otimes \bw \left\langle
y \right\rangle \right)_m = A_{m-1}y$.
We will use the following lemma several times in the rest of the  proof.
\begin{lemma}\label{first}
The element $\psi(\alpha_1)\phi(\alpha_2)\cdots \psi(\alpha_m)$ in $A_{m-1}y$ is non-zero. As
a consequence the elements $\psi(\alpha_1)$, \dots, $\psi(\alpha_m)$ in
$A_1$ are linearly independent. In particular, the restriction $ \psi|_{\la^*}$ is injective.
\end{lemma}
\begin{proof}[Proof of Lemma]
It follows from~\eqref{dec} that the Chevalley-Eilenberg differential
$$d_{m-1}\colon \bw^{m-1} \la^* \to \bw^m \la^*$$ is zero. Thus $H^m(\bw\la^*) =
\bw^m \la^*$. In particular, $\alpha_1\wedge\dots \wedge \alpha_m$ is a
non-zero element in $H^m(\bw\la^*)$. Since $\psi$ is a quasi-isomorphism we get
that $[\psi(\alpha_1\wedge \dots \wedge \alpha_m)]$ is a non-zero element in
the $m$th cohomology group of $A\otimes \bw \left\langle y \right\rangle$.
But then $\psi(\alpha_1\wedge \dots \wedge \alpha_m) =
\psi(\alpha_1)\psi(\alpha_2)\cdots \psi(\alpha_m)$ is a non-zero element of
$A_{m-1}y$. This proves the first claim of the lemma.

Now suppose there is a $j$ such that
\begin{equation*}
\psi(\alpha_j) = \sum_{i\not=j} a_i \psi(\alpha_i)
\end{equation*}
for some real numbers $a_i\in \R$. Then, since $\psi(\alpha_i)^2=0$ in
$A\otimes \bw \left\langle y \right\rangle$, we get that
$\psi(\alpha_1)\psi(\alpha_2)\cdots \psi(\alpha_m)=0$. Thus we got a
contradiction to the already proved fact. This shows that the elements
$\psi(\alpha_1)$, \dots, $\psi(\alpha_m)$ are linearly independent in $A_1$.
\end{proof}
Now we resume the proof of the theorem.
We will distinguish two cases: the first when $z=0$ and the second
when $z\not=0 $.

For $z=0$ we have $l=0$ and thus we need to show that $G\cong \R^{m}$ or
equivalently that $\la$ is an abelian Lie algebra. Notice that $\la$ is abelian
if and only if $d_1^{CE}$ is zero. Thus it is enough to check that  $\ker(d_1^{EC}) =
\la^*$.

  Since $z=0$ the differentials in the complex
$A\otimes \bw \left\langle y \right\rangle$  are zero. Therefore, its first
cohomology group coincides with its component of degree one
\begin{equation*}
\left( A\otimes \bw \left\langle y \right\rangle \right)_1 = A_1 \oplus A_0 y.
\end{equation*}
As $\psi$ is a quasi-isomorphism it induces the isomorphism
\begin{equation*}\label{asdf}
\begin{aligned}
  {[\psi]}\colon \ker(d_1^{CE}) = H^1(\bw\la^*)& \to A_1\oplus A_0 y\\
\alpha & \mapsto \psi(\alpha).
\end{aligned}
\end{equation*}
Thus we get the commutative diagram
\begin{equation*}
\xymatrix{
\ker\left( d_1^{CE} \right) \ar[r]^-{[\psi]}_-{\cong} \ar@{^{(}->}[d]  & A_1 \oplus A_0 y \\
{\ \ \la^*\ \ } \ar@{>->}[ru]_{\psi|_{\la^*}} &
}
\end{equation*}
where $\psi|_{\la^*}$ is injective by Lemma~\ref{first}.
Thus we get that the isomorphism $[\psi]$ is the composition of two
monomorphisms. But this is possible only if both of them are isomorphisms as
well. Therefore $\ker(d_1^{CE}) = \la^*$ as required.

Now we assume that $z\not=0$. In this case the first differential in $A\otimes
\bw \left\langle y \right\rangle$ is given by
\begin{equation*}
d_1(a_1+ a_0 y) = a_0 z,
\end{equation*}
where $a_0 \in A_0 = \R$. Thus $a_1 + a_0 y$ is in the kernel of $d_1$ if and
only if $a_0=0$. Moreover $d_0=0$ in $A\otimes \bw \left\langle y
\right\rangle$. Thus we get
\begin{equation*}
H^1\left(A\otimes \bw \left\langle y \right\rangle\right) = A_1.
\end{equation*}
As $\psi$ is a quasi-isomorphism the induced map
\begin{equation*}\label{adf}
\begin{aligned}
  {[\psi]}\colon \ker(d_1^{CE}) = H^1\left(\bw\la^*\right)& \to A_1\\
\alpha & \mapsto \psi(\alpha).
\end{aligned}
\end{equation*}
is an isomorphism of vector spaces.
It follows from Lemma~\ref{first} that $\psi(\alpha_1)$, \dots,
$\psi(\alpha_{b_1})$ is a basis of $A_1$.

Now we will show that $b_1=m-1$. Suppose $b_1=m$. Then
\begin{equation*}
\psi(\alpha_1) \psi(\alpha_2) \dots \psi(\alpha_m) \in A_m =0
\end{equation*}
and thus the above product must be zero, which contradicts to Lemma~\ref{first}.
Now assume $b_1\le m-2$. Then
\begin{equation*}
 \dim \la^*=m > b_1 +1 = \dim A_1 + 1 = \dim (A_1 \oplus A_0 y).
\end{equation*}
But this is impossible as $\psi|_{\la^*}\colon \la^* \to A_1 \oplus A_0y$ is
a monomorphism by Lemma~\ref{first}.
Thus $b_1=m-1$ and
\begin{equation*}
 \dim \la^*=m = b_1 +1 =\dim (A_1 \oplus A_0 y).
\end{equation*}
This equality together with Lemma~\ref{first} imply that $\psi|_{\la^*}$ is an
isomorphism.
Now without loss of generality we can assume that $\alpha_m
=(\psi|_{\la^*})^{-1} (y) $.
Further there is a basis $\beta_1$, \dots, $\beta_{m-1}$ of $\ker(d_1^{CE})$
such that
\begin{equation*}
d_1^{CE} \alpha_m = \sum_{j=1}^{r} \beta_{2j-1} \wedge \beta_{2j}
\end{equation*}
for some natural number $r\le \left\lfloor \frac{m-1}2 \right\rfloor$.
Now to prove that $\la \cong \heisla(1,l) \oplus \abla_{m-2l-1} $ it remains
to show
that $r=l$.
We have
\begin{equation*}
z^r = (dy)^r = (d\psi(\alpha_m))^r = \psi \left( (d_1^{CE}\alpha_m)^r \right) =
r! \psi(\beta_1) \dots \psi(\beta_{2r}) \not=0
\end{equation*}
by Lemma~\ref{first}. Further
\begin{equation*}
z^{r+1} = \psi\left( (d_1^{CE}\alpha_m)^{r+1} \right) =0.
\end{equation*}
Thus $r=l$ as claimed and this finishes the proof.
\end{proof}
\begin{remark}
Note, that the above theorem can be also restated in more topological terms.
Namely,  an $m$-dimensional aspherical nilpotent manifold $M$ admits an
almost formal model of dimension $m$ and index $l$ if and only if
\begin{equation*}
\pi_1(M) \otimes \R \cong \heis(1,l) \times \R^{m-2l-1}.
\end{equation*}
\end{remark}
{Next}, using Theorem~\ref{heisenberg} we are able to classify quasi-Sasakian
and quasi-Vaisman compact nilmanifolds.
\begin{theorem}\label{classification_qS}
The $(2n+1)$-dimensional compact nilmanifold $\Gamma\backslash G$ admits a quasi-Sasakian structure
of index $l$
if and only if $G$ and $\heis(1,l)\times \R^{2(n-l)}$ are isomorphic as Lie
groups. \end{theorem}
\begin{proof}
In Section~\ref{quasi-Sasakian},
it was explained how to construct a quasi-Sasakian structure on
\begin{equation*}
\Gamma\backslash\left(\heis(1,l)\times \R^{2(n-l)}\right)
\end{equation*}
for any cocompact subgroup $\Gamma$ of $\heis(1,l)\times \R^{2(n-l)}$.  It is easy to prove that this structure has index $l$.

Now, suppose $M:=\Gamma\backslash G$ is a nilmanifold that admits a quasi-Sasakian structure
such that $[d\eta]_B^l\not=0$ and $[d\eta]_B^{l+1} = 0$. Then the almost formal
model \eqref{eq:model_qS} is of  index $l$. Therefore by
Theorem~\ref{heisenberg}, we have that $G$ and $H(1,l)\times \R^{2(n-l)}$ are
isomorphic.
\end{proof}
\begin{theorem}\label{classification_pV}
The $(2n+2)$-dimensional compact nilmanifold $\Gamma\backslash G$ admits a quasi-Vaisman structure
if and only if $G$ is isomorphic to $\heis(1,l)\times \R^{2(n-l)+1}$ as a Lie
group. Moreover, in this case,
\begin{equation*}
[d\eta]_B^l\not=0,\quad [d\eta]_B^{l+1}=0.
\end{equation*}
\end{theorem}
\begin{proof}
In Section~\ref{quasi-Vaisman}
it was shown that every nilmanifold modelled on $H(1,l)\times \R^{2(n-l)+1}$
admits a quasi-Vaisman structure. It is easy to check that
this structure has index $l$.

Now, suppose $M=\Gamma\backslash G$ admits a quasi-Vaisman structure of index $l$. Then
by Theorem~\ref{model_pV} it has an almost formal model of index $l$. Applying Theorem~\ref{heisenberg}, we get that $G$ and $H(1,l)\times
\R^{2(n-l)+1}$ are isomorphic.
\end{proof}

Note that using Remark \ref{index-Sasakian} and Theorem \ref{classification_qS}, we directly deduce a result which was initially proved in \cite{nilsasakian}.
\begin{corollary}
The $(2n+1)$-dimensional compact nilmanifold $\Gamma\backslash G$ admits a Sasakian structure
if and only if $G$ and $\heis(1,n)$ are isomorphic as Lie
groups.
\end{corollary}
Finally, using Remark \ref{index-Vaisman} and Theorem \ref{classification_pV}, we directly deduce another result which has been proved recently in \cite{bazzoni}.
\begin{corollary}
The $(2n+2)$-dimensional compact nilmanifold $\Gamma\backslash G$ admits a Vaisman structure
if and only if $G$ and $\heis(1,n)\times \R$ are isomorphic as Lie
groups.
\end{corollary}
\section{Mapping torus and solvmanifolds}

Let $M$ be a compact Riemannian manifold with Riemannian metric $h$ and $f: M \to M$ be an isometry. Suppose that $a$ is a positive constant and consider the mapping torus
\[
M_{(f, a)} = (M \times \mathbb{R}) / \rho_{(f, a)},
\]
with $\rho_{(f, a)}: \mathbb{Z} \times (M \times \mathbb{R}) \to M \times \mathbb{R}$ the action of the discrete subgroup $\mathbb{Z}$ defined by
\[
\rho_{(f, a)}(k, (x, t)) = (f^k (x), t + ak).
\]
Let $dt \otimes dt$ be the standard flat metric on $\mathbb{R}$ and $g$ the product metric on $M \times \mathbb{R}$
\[
g = h + dt \otimes dt.
\]
Then, it is clear that the $\mathbb{Z}$-action is isometric. So, $g$ induces a Riemannian metric on the mapping torus $M_{(f, a)}$.

Now, denote by $U$ the vector field on $M_{(f, a)}$ induced by the $\rho_{(f, a)}$-invariant vector field $\frac{\partial}{\partial t}$ on $M \times \mathbb{R}$.
Since the vector field $U$ is unitary and parallel, the following theorem is a
direct consequence of~\cite[Corollary~3.2]{hltvaisman}.
\begin{theorem}\label{theorem-appendix}
Let $(M, h)$ be a compact Riemannian manifold and $H^k(M_{(f, a)})$ the de Rham cohomology group of order $k$ of the mapping torus of $M$ by the isometry $f: M \to M$ and the positive constant $a$. Then,
\[
H^k(M_{(f, a)}) \simeq H^k_U(M_{(f, a)}) \oplus H^{k-1}_U(M_{(f, a)}),
\]
where $H^*_U(M_{(f, a)})$ is the basic cohomology of $M_{(f, a)}$ with respect to the foliation generated by the vector field $U$ induced by the invariant vector field $\frac{\partial}{\partial t}$ on $M \times \mathbb{R}$. Moreover,
\[
H^k_U(M_{(f, a)}) \simeq \left\{ \alpha \in \Omega^k(M) \,\middle|\, \alpha
\mbox{ is h-harmonic and } f^*\alpha = \alpha \right\}.
\]
\end{theorem}
\subsection{Mapping torus and quasi-Vaisman manifolds.}
\label{mappingtorusquasivaisman}
It is well known that if $M$ is a Sasakian manifold
then a mapping torus of $M$ with respect to any Sasakian automorphism
can be endowed with a Vaisman structure.

Now, we show that a similar relation holds between quasi-Sasakian and quasi-Vaisman
manifolds.
Let $(M,{\varphi},\xi,\eta,h)$ be a quasi-Sasakian manifold.
We start by constructing a quasi-Vaisman structure on $M\times \R$.
Write $dt$ for the volume form on $\R$.
Define the metric on $M\times \R$ by
\begin{equation*}
g = h + dt\otimes dt
\end{equation*}
and the complex structure $J$ by
\begin{equation*}
J = {\varphi} - \frac{\partial}{\partial t} \otimes \eta + \xi \otimes dt,
\end{equation*}
that is
\begin{equation*}
J\left(X,a\frac{\partial}{\partial t}\right) = \left({\varphi} X + a \xi, - \eta(X)
\frac{\partial}{\partial t}\right).
\end{equation*}
\begin{proposition}
\label{qV}
The manifold $(M\times \R,J,g,dt)$ with $J$ and $g$ defined above is a
quasi-Vaisman manifold. Also the manifold $ M\times S^1 \cong \left( M\times \R
\right)/\Z$ inherits
a quasi-Vaisman structure from $M\times \R$.
\end{proposition}
\begin{proof}
 As $({\varphi},\xi,\eta,h)$ is a normal almost contact metric structure,
we have that  $\left( J,g \right)$ is Hermitian (see
\cite[Section~6.1]{blair2}).
Moreover, its fundamental $2$-form is given by $\Omega := g \circ( \id \otimes
J)$.
As  $dt\circ J = -\eta$,
we get
\begin{equation*}
g\circ (\id \otimes J) = \Phi + \eta  \otimes dt - dt \otimes \eta,
\end{equation*}
 where $\Phi$ is the fundamental $2$-form of the almost contact metric structure $(\varphi, \xi, \eta, h)$.
Thus
\begin{equation*}
\Omega = \Phi + \eta \otimes dt - dt \otimes \eta =   \Phi +  \eta\wedge dt
\end{equation*}
and
\begin{equation*}
d\Omega = d\eta \wedge dt .
\end{equation*}
It is clear that $U=\frac{\partial}{\partial t}$  is parallel unit vector field on
$M\times \R$.
It is left to show that $\lie_U J = \lie_{\frac{\partial}{\partial t}} J=0$.
Now, we have
\begin{equation*}
\begin{aligned}
(\lie_{U} J) \left( X , aU\right) = \lie_U \left( \varphi X + a\xi, -\eta(X) U
\right) - J \left( \lie_U \left( X , aU \right) \right) = 0,
\end{aligned}
\end{equation*}
where $X\in \vf\left( M \right)$ and $a\in \R$. This proves that $\lie_U J = 0$.
\end{proof}
Now let
 $f\colon M\to M$ be  a
diffeomorphism that preserves the almost contact metric structure.
For every positive constant $a\in \R$, we consider the mapping torus $M_{(f, a)}$ of $M$ by $f$ and $a$, that is,
\[
M_{(f, a)} = (M \times \mathbb{R}) /\Z,
\]
where the action $\rho$ of the discrete subgroup $\Z$ on $M \times \mathbb{R}$ is given by
\[
\rho(k, (x, t)) = (f^k(x), t + k a).
\]
The above action preserves the quasi-Vaisman structure on $M\times
\R$. Therefore the quotient compact smooth manifold
\begin{equation*}
M_{(f,a)} = \left( M\times \R \right)/\Z
\end{equation*}
is quasi-Vaisman.


On the other hand, let $(M,J,g,\theta)$ be a quasi-Vaisman manifold. Since $\theta$ is a closed
form, the foliation $\ker \theta = \left\langle U \right\rangle^\perp$ is
integrable.   Let $L$ be a leaf of this foliation.
Notice that if $M$ is a mapping torus of a quasi-Sasakian manifold $N$, then 
$L$ is isometric to $N$ and thus it is quasi-Sasakian. We are going to show that
this holds also for a general quasi-Vaisman manifold.

Define an almost contact structure
$(\phi,\xi,\eta,g)$ on $L$ where $\xi=V$, $\eta$  and $g$ are given by
restricting from $M$, and
\begin{equation*}
\phi \xi :=0,\quad \phi X = JX, \mbox{ for } X \in \left\langle U,V
\right\rangle^\perp.
\end{equation*}
\begin{proposition}
\label{leafqV}
The almost contact structure $(\phi,\xi,\eta,g)$ on $L$ is quasi-Sasakian.
\end{proposition}
\begin{proof}
We have to check that $N_\phi + d\eta \otimes \xi=0$ and $d\Phi=0$, where~$\Phi (X,Y) = g (X,\phi Y)$.

For every $W\in \left\langle U \right\rangle^\perp$, we have
\begin{equation}\label{phiW}
\phi(W) = \phi\left( W -\eta(W) V \right) = J(W -\eta(W) V) = JW +
\eta(W) U.
\end{equation}

For $X$, $Y\in \left\langle U,V \right\rangle^\perp$, we obtain by
applying~\eqref{phiW} several times
\begin{equation*}
\begin{aligned}
 N_\phi (X,Y) &  = \phi^2 \left[ X,Y \right] 
 -\phi (\left[ \phi X,Y \right] + \left[ X,\phi Y \right])  
 + \left[ \phi X, \phi Y \right]  \\[0.7ex] &=
N_J\left( X,Y \right) + \eta \left( \left[ X,Y \right] \right) V +
\eta\left( J\left[ X,Y \right] - \left[ JX,Y \right]  - \left[ X,JY
\right]\right) U.
\end{aligned}
\end{equation*}
Since $X$ and $Y$ are orthogonal to $V$, we have $\eta \left( \left[ X,Y
\right] \right) = -d\eta \left( X,Y \right)$. Further, using integrability, the
last term in  the above formula can be written as $\eta \left(J\left[ JX,JY
\right]  \right) U$. Since $X$ and $Y$ are orthogonal to $V$, it follows that
$JX$ and $JY$ are orthogonal to $U$. As the foliation $\left\langle U
\right\rangle^\perp$ is integrable, we get that $\left[ JX,JY \right]$ is
orthogonal to $U$, and thus $J\left[ JX,JY \right]\in \left\langle V
\right\rangle^\perp$. This shows that $\eta\left( J\left[ JX,JY \right] \right)
=0$. Therefore $N_\phi \left( X,Y \right) = - d\eta \left( X,Y \right)V$ for
$X$, $Y\in \left\langle U, V \right\rangle^\perp$.

Now, suppose that $X\in \left\langle U,V \right\rangle^\perp$. Then, applying~\eqref{phiW}, we get
\begin{equation*}
\begin{aligned}
N_\phi\left( X,V \right) = -\left[ X,V \right] - J\left[ JX,V \right]
+\eta\left( \left[ X,V \right] \right) V - \eta\left( \left[ JX,V \right] -
J\left[ X,V \right] \right) U.
\end{aligned}
\end{equation*}
From integrability of $J$, we have
\begin{equation*}
\begin{aligned}
\left[ X,V \right] + J\left[ JX,V \right] &= \left( \lie_U J \right)(X)\\
\left[ JX,V \right] - J\left[ X,V \right] &= \left( \lie_U J \right) \left( JX
\right).
\end{aligned}
\end{equation*}
As, by definition of quasi-Vaisman manifold $\lie_U J=0$, we get
that $N_\phi \left( X,V \right) = -d\eta \left( X,V \right)V$.

It is left to show that $d\Phi =0$. For any $X$, $Y$, $Z\in \left\langle U
\right\rangle^\perp$, we have $d\Phi \left( X,Y,Z \right) = d\Omega \left( X,Y,Z
\right) = \left( d\eta \wedge \theta \right)\left( X,Y,Z \right) =0$ as
$\theta\left( X \right) = \theta\left( Y \right) = \theta\left( Z \right) =0$.
\end{proof}

Now, we show that starting with a quasi-Vaisman manifold we can
construct a new quasi-Sasakian manifold by using mapping torus construction.

Suppose that $(M^{2n+2}, J, g,\theta)$ is a quasi-Vaisman manifold with
quasi-anti-Lee $1$-form $\eta$,  quasi-Lee vector field $U$, and  quasi-anti-Lee
vector field $V$. Then, on the product manifold $M \times \mathbb{R}$ we
consider the almost contact metric structure $(\varphi, \xi, \eta, h)$, where
$h$ is the metric product of $g$ and the standard metric  on $\mathbb{R}$, that is,
\begin{equation*}\label{metric-quasi-Sasakian}
h = g + dt \otimes dt,
\end{equation*}
the Reeb vector field $\xi=V$, $\eta = h(\xi,-)$, and the $(1, 1)$-tensor field $\varphi$ is given by
\begin{equation*}\label{varphi-xi-quasi-Sasakian}
\begin{aligned}
&\varphi U = \ddt,\quad \varphi \ddt = -U,\quad \varphi V =0,\\[2ex] &\varphi X=
JX,\quad \mbox{for } X \in \left\langle U,V \right\rangle^{\perp}.
\end{aligned}
\end{equation*}

\begin{proposition}\label{Vaisman-R-circle}
The almost contact metric structure $(\varphi, \xi, \eta , h)$ on $M \times
\mathbb{R}$ is quasi-Sasakian.
\end{proposition}
\begin{proof}
As $U$ is parallel, from Proposition~\ref{leafqV} it follows that $M$ is a local product of a
real line and a quasi-Sasakian manifold. Therefore $M\times \R$ is a local
product of a quasi-Sasakian manifold and $\R^2$, where the distribution
$\mathcal{D}$ tangent
to the  {2-dimensional} factor is generated by $U$ and $\ddt$.
Clearly, the restriction $J_{\mathcal D}$ of $\phi$ to $\mathcal{D}$ is an integrable almost
complex structure on $\mathcal{D}$.
For $X$, $Y\in \mathcal{D}$, we define $\omega\left( X,Y \right) = h\left( X,J_{\mathcal D}Y
\right)$. We get
\begin{equation*}
\omega \left( U, \ddt \right) = h\left( U, J_{\mathcal D}\ddt \right) = h(U, -U) = -1.
\end{equation*}
This implies that $\omega$ is a volume form on every
leaf of $\mathcal{D}$. Thus the leaves of $\mathcal{D}$ are K\"ahler.
Therefore, $M\times \R$ is a local product of quasi-Sasakian and a K\"ahler
manifolds, and hence it is a quasi-Sasakian manifold \it{per se}.
\end{proof}
Let $\left( M,J,g,\theta \right)$ be a Vaisman manifold.
Suppose $f\colon M\to M$ is an isometry that preserves the quasi-Vaisman structure.
Then for every $a\in \R_{>0}$, the mapping torus $M_{(f,a)}$ inherits a
quasi-Sasakian structure from $M\times \R$.

\subsection{Model of a mapping torus}
In this section, we will describe a model for a mapping torus by an isometry.
\begin{proposition}
\label{mappingtorusmodel}
Let $(M,g)$ be a compact  Riemannian manifold,  $f$ an isometry of $M$, and $a$ a
positive real
number.
Then the  CDGA $\left( \Omega^*(M)^f\otimes \bw \left\langle y \right\rangle,\ dy =0 \right)$
is a model of the mapping torus $M_{(f,a)}$, where
\[
\Omega^*(M)^f = \{ \alpha \in \Omega^*(M) \;|\; f^*\alpha = \alpha \}.
\]
\end{proposition}
\begin{proof}
Denote by $\xi$ the vector field on $M_{(f,a)}$ induced by the vector field
$\partial/\partial t$ on $M\times \R$.
Since $\xi$ is unitary and parallel, by Theorem~\ref{invariant}, the inclusion
$\Omega^*_{\lie_\xi} \left( M_{(f,a)} \right) \hookrightarrow
\Omega^*(M_{(f,a)})$ is a
quasi-isomorphism of CDGAs.

Denote by
  $\theta$ the  metric dual of $\xi$. By
Corollary~\ref{alg-conn-orthonormal}, the map
\begin{equation*}
\begin{aligned}
\chi  \colon \R & \to \Omega^1(M_{(f,a)}) \\
s  & \mapsto  s\theta
\end{aligned}
\end{equation*}
is an algebraic connection for the locally free action of $\R$ on $M_{(f,a)}$
induced by~$\xi$. Since $d\theta=0$, by Theorem~\ref{chevalley} the CDGA
\begin{equation*}
\left( \Omega^*_B(M_{(f,a)},\xi) \otimes \bw \left\langle y \right\rangle, dy =
0 \right)
\end{equation*}
is quasi-isomorphic to $\Omega^*_{\lie_\xi}(M_{(f,a)})$ and thus to
$\Omega^*(M_{(f,a)})$.

Define $\hat{f}\colon M\times \R \to M\times \R$ by  $\hat{f} (m,t) =
(f(m), t+a)$. Then the following homomorphisms of CDGAs induced by pull-backs
of differential forms
\begin{equation}
\label{isochain}
\begin{aligned}
\Omega^*_B\left( M_{(f,a)},\xi \right) \xrightarrow{\pi^*} \Omega^*_B\left( M\times \R,
\partial/\partial t
\right)^{\hat{f}} \xleftarrow{pr_1^*} \Omega^*(M)^f
\end{aligned}
\end{equation}
are isomorphisms of CDGAs.
Tensorising~\eqref{isochain} with $\bw\left\langle y \right\rangle$ and
defining $dy=0$ on all resulting graded algebras,  we get
the isomorphisms of CDGAs.
 \end{proof}
\subsection{Mapping torus and a semi-direct product. }
Let $G$ be a Lie group and $\Gamma$ a cocompact subgroup of $G$. Consider an
action $\phi_a \colon G\to G$, $a\in \R$ of $\R$ on $G$. Then the product on
the
semi-direct product $G\rtimes_\phi \R$ is given by
$(g,t)(g',t') = (g\phi_t(g'), t+t')$.
Suppose there is $a\in \R_{>0}$ such that $\phi_a (\Gamma) = \Gamma$. Then
the group $a\Z$ acts on $\Gamma$ and the semi-direct product $\Gamma \rtimes_\phi a\Z $
can be considered as a cocompact discrete subgroup of $G\rtimes_\phi \R$.

Moreover, the action of $a\Z$ on $G$ descends to an action  on
$\Gamma\backslash G$. Therefore, we can consider the mapping torus
$\left(\Gamma\backslash G
		\right)_{(\phi_a,a)}$.
We have the following identification.
\begin{proposition}
\label{mappingtorussemidirectproduct}
Suppose $G$ is endowed with a Riemannian metric $g$, which is invariant under
the action of $\Gamma$, and all $\phi_a$ are isometries. Then
the mapping torus $\left(\Gamma\backslash G \right)_{(\phi_a,a)}$ and
the quotient
\begin{equation}
\label{gammaphi}
 \left( \Gamma\rtimes_\phi
a\Z \right) \backslash \left( G\rtimes_\phi \R \right)
\end{equation}
are isometric Riemannian manifolds.
\end{proposition}
\begin{proof}
The mapping torus $\left( \Gamma\backslash G \right)_{(\phi_a,a)}$ is by
definition the quotient of $\left( \Gamma\backslash G \right)\times \R$ under
the action $\rho$ of $a\Z$ defined by $\rho_{ak} ([g],t) = ([\phi^k_a(g)],
ak+t)  = ([\phi_{ka}(g)],ak + t) $.  We can identify $\left( \Gamma\backslash G \right)\times \R$
with $\Gamma \backslash \left( G\times \R \right)$.
Thus $\left( \Gamma\backslash G \right)_{(\phi_a,a)}$ is isometric to the
iterated quotient $\left( a\Z \right) \backslash \left( \Gamma \backslash \left(
G\times \R \right) \right)$.

Now, as $\Gamma$ is a normal subgroup of $\Gamma \rtimes_\phi a\Z$, by two step
reduction, the manifold~\eqref{gammaphi}
is isometric to the quotient of $\Gamma\backslash \left( G\rtimes_\phi
\R
\right)$ under the action of  the group $\Gamma\backslash \left( \Gamma \rtimes_\phi a\Z
\right) \cong a\Z$ given by
\begin{equation*}
[(\gamma,ak)] \cdot [(g,t)] = [ (\gamma,ak)\cdot (g,t)] = [ (\gamma
\phi_{ak} (g), ak +t) ] = [ (\phi_{ak}(g), ak+t) ] .
\end{equation*}
Now one can see that both $\left( \Gamma\backslash G \right)_{(\phi_a,a)}$ and
\eqref{gammaphi} are quotients of $\Gamma\backslash \left( G\times \R \right)$
under the same action of $a\Z$. Hence they are isometric.
\end{proof}

\subsection{Examples of quasi-Sasakian and  quasi-Vaisman manifolds.}
In this section we give some explicit examples of solvmanifolds that admit
a quasi-Sasakian or a quasi-Vaisman structure. We also exhibit almost
formal models for them.

\begin{example}
\label{ex:vaisman}
We start by reviewing a construction of a Vaisman solvmanifold considered
in~\cite{MaPa}.
Consider the Heisenberg group $H(1,1)$. In the notation of
Section~\ref{quasi-Sasakian}, its
standard left-invariant Sasakian structure $(\varphi_H, \xi_H, \eta_H, h_H)$ is given by
\[
\varphi_H = \alpha_1 \otimes X_2 - \alpha_2 \otimes X_1, \; \; \xi_H = X_3, \;
\; \eta_H = \alpha_3,\quad
h_H = \sum_{i=1}^{3} \alpha_i \otimes \alpha_i.
\]
Further, define $\Gamma\left( 1,1 \right)$ as  the subgroup of $H\left( 1,1
\right)$ consisting of the matrices in $H\left( 1,1
\right)$ with integer entries. It is obviously cocompact.
Denote by $L$ the compact Sasakian manifold $\Gamma(1,1)\backslash H\left( 1,1
\right)$.

Consider the action $F$ of $\R$ on $H(1,1)$   given by
\begin{equation*}
\begin{aligned}
F_u (p, q, t) & = (p \cos u + q \sin u, -p \sin u + q \cos u,
 t  + \frac{1}{4} (q^2 - p^2)\sin(2u) - p q \sin^{2}u).
\end{aligned}
\end{equation*}
It acts by group automorphisms and preserves the standard left-invariant
Sasakian structure (see~\cite{MaPa}).  For $u=\frac\pi{2}$, we have
\begin{equation*}
F_{\frac\pi{2}} \left( p,q,t \right)  = (q,-p, t -pq).
\end{equation*}
So the cocompact subgroup $\Gamma\left( 1,1 \right)$ of $H\left( 1,1
\right)$ is preserved by  $F_{\frac\pi2}$. Hence we can consider the Vaisman
manifold $M :=  L_{(f,\frac\pi2)}$ with $f=F_{\frac\pi2}$.
By Proposition~\ref{mappingtorussemidirectproduct}, $M$ is a solvmanifold
modelled on the solvable group $G:=H\left( 1,1 \right)\rtimes_{F} \R$. Notice
that $G$ is not nilpotent and even not a completely solvable Lie
group.

Now we will find an almost formal model of $M$.
By
Proposition~\ref{mappingtorusmodel}, the CDGA
\begin{equation}
\label{omegalf}
\left( \Omega^*\left( L \right)^f\otimes \bw\left\langle y \right\rangle, dy=0
\right)
\end{equation}
is a model of $M$. We will use the following remark.
\begin{remark}
\label{qicdgas}
If $A$ and $B$ are quasi-isomorphic CDGAs, then $\left( A\otimes \bw\left\langle
y \right\rangle, dy=0 \right)$ and $\left( B\otimes \bw\left\langle y
\right\rangle, dy=0  \right)$ are quasi-isomorphic as well.
\end{remark}
Thus to simplify~\eqref{omegalf}, we can replace $\Omega^*(L)^f$ with a
quasi-isomorphic CDGA.
Since $L$ is a nilmanifold,  by  Nomizu theorem, the inclusion $j\colon \bw \heisla\left( 1,1
\right)^* \to \Omega^*(L)$, given by left-invariant extension, is a
quasi-isomorphism of CDGAs. The image of $j$ is generated by $1$-forms
$\alpha_1$, $\alpha_2 $, $\alpha_3 $. By easy computation
\begin{equation*}
f^* \alpha_1  = \alpha_2,\quad f^* \alpha_2  = -\alpha_1 ,\quad f^*\alpha_3  =
\alpha_3 .
\end{equation*}
In particular, the image of $j$ is invariant under the action of $f$.

Notice that the automorphism $f$ of $\Omega^*(L)$ is of finite order, namely
$f^4=\id$.
Therefore the map
\begin{equation*}
\hat\jmath\colon \left( \bw h\left(1,1 \right)^* \right)^{f^*} \to
\Omega^*(L)^f
\end{equation*}
induced by $j$ is a quasi-isomorphism.
The image of $\hat\jmath$ has a  basis
\begin{equation*}
1,\quad \alpha_3, \quad \alpha_1  \wedge \alpha_2,\quad \alpha_1 \wedge \alpha_2 \wedge
\alpha_3 .
\end{equation*}
The de Rham differential gives zero on all elements of this basis, except
$\alpha_3 $:
\begin{equation*}
d\alpha_3  = d\left( dt -pdq \right) = - dp\wedge dq = -\alpha_1 \wedge
\alpha_2 .
\end{equation*}
Thus we can conclude that $\Omega^*(L)^f$ is quasi-isomorphic to the CDGA
\begin{equation}
\label{balpha}
\left( B\otimes \bw \left\langle \alpha_3  \right\rangle, d\alpha_3  =
-\alpha_1 \wedge \alpha_2  \right),
\end{equation}
where $B$ is a CDGA with zero differential and  $B_0 =\left\langle 1 \right\rangle$, $B_2 =
\left\langle \alpha_1 \wedge \alpha_2  \right\rangle$, $B_k=0$ for
$k\not=0$ or $2$. Substituting~\eqref{balpha} in~\eqref{omegalf} instead of
$\Omega^*(L)^f$, we obtain the model
\begin{equation}
\label{balphay}
\left( B\otimes \bw\left\langle \alpha_3 ,y \right\rangle, d\alpha_3  =
-\alpha_1  \wedge \alpha_2,\ dy=0 \right)
\end{equation}
of  $M$. If we define $A= \left( B\otimes \bw\left\langle y \right\rangle, dy=0
\right)$, then~\eqref{balphay} becomes
\begin{equation}
\label{aalpha}
\left( A \otimes \bw\left\langle \alpha_3  \right\rangle,\ d\alpha_3  =
-\alpha_1  \wedge \alpha_2  \right).
\end{equation}
Notice that $A$ has a zero differential, and therefore~\eqref{aalpha} is an
almost formal model. Using Theorem \ref{theorem-appendix} and carefully examining the quasi-isomorphisms we used, one
can show that in fact $A$ is isomorphic to the basic cohomology algebra
$H_B^*(M,V)$ and $B$ is isomorphic to the basic cohomology algebra $H^*_B\left(
M, \left\langle U,V \right\rangle \right)$.
\end{example}
\begin{example}
Let $G = H(1,1)\rtimes_F \R$ with the action $F$ defined above and $\Gamma =
\Gamma(1,1) \rtimes_F \frac\pi2 \Z$ its cocompact subgroup. Then the Vaisman
manifold  $M$
constructed in Example~\ref{ex:vaisman} is isomorphic to $\Gamma\backslash G$.
We consider the trivial action of $\R$ on $G$. By
Proposition~\ref{mappingtorussemidirectproduct}, the corresponding mapping torus
$N:= M_{(\id,1)}$ is a solvmanifold modelled on the solvable group $G\times \R$.
By Proposition~\ref{Vaisman-R-circle} and the discussion thereafter, $N$ has a
quasi-Sasakian structure induced by the Vaisman structure on~$M$.
Hence $N$ is an example of a quasi-Sasakian solvmanifold.

By  Proposition~\ref{mappingtorusmodel}, the CDGA
\begin{equation}
\label{omegam}
\left( \Omega^*(M) \otimes \bw \left\langle z \right\rangle,\ dz = 0 \right)
\end{equation}
is a model of $N$. By Remark~\ref{qicdgas} we can replace $\Omega^*(M)$
in~\eqref{omegam} with any model of $M$. Therefore,
substituting~\eqref{aalpha} in~\eqref{omegam}, we get
that the CDGA
\begin{equation}
\label{aalphaz}
\left( A\otimes \bw\left\langle \alpha_3 ,z \right\rangle, d\alpha_3 =-\alpha_1
\wedge \alpha_2, dz=0 \right)
\end{equation}
is a model of $N$. Defining $C=\left( A\otimes \bw\left\langle z
\right\rangle, dz=0 \right)$, we can write~\eqref{aalphaz} as
\begin{equation}
\label{calpha}
\left( C\otimes \bw\left\langle \alpha_3  \right\rangle, d\alpha_3  = -\alpha_1
\wedge \alpha_2  \right).
\end{equation}
As $A$ has a zero differential, the same is true for $C$. Hence~\eqref{calpha}
is an almost formal model of $N$.
\end{example}
\begin{example}
Now we construct an example of a  quasi-Vaisman solvmanifold.
We start with the nilpotent Lie group $G = H(1,1)\times \R^2$, which is the
group $G$ defined in Section~\ref{quasi-Sasakian} with $l=1$ and $n=2$.
As it was shown there,  $G$ admits a left-invariant quasi-Sasakian
structure. Denote by $\Gamma$ the subgroup $\Gamma(1,1) \times \Z^2$ of
$G$ and by $N$ the resulting compact quasi-Sasakian nilmanifold  $\Gamma\backslash
G$.

We define the action $A$ of $\R$ on $G = H(1,1) \times \R^2$ by
\begin{equation*}\label{A-representation}
\begin{aligned}
A_u ( (p, q, t), (x, y) )  =   \left( F_u(p,q,t),  (x \cos u + y \sin u, -x \sin u +
y \cos u)\right).
\end{aligned}
\end{equation*}
For $u=\frac\pi2$, we get
\begin{equation*}
A_{\frac\pi2} ( (p,q,t), (x,y)) = ( (-q,p,t-pq) , (y,-x)).
\end{equation*}
Denote $A_{\frac\pi2}$ by $\hat{f}$. Then by
Proposition~\ref{mappingtorussemidirectproduct}, the mapping torus
$N_{(\hat{f}, \frac\pi2)}$ is a solvmanifold modelled on $G\rtimes_A \R$.
Notice that $G\rtimes_A \R$ is not nilpotent and even not completely solvable.
 By the discussion after Proposition~\ref{qV}, the mapping torus
$N_{(\hat{f},\frac\pi2)}$ has a natural quasi-Vaisman structure induced by the
quasi-Sasakian structure on $N$.

Now we compute an almost formal model of $N_{(\hat{f},\frac\pi2)}$.
By Proposition~\ref{mappingtorusmodel}, the CDGA
\begin{equation}
\label{omegan}
\left( \Omega^*(N)^{\hat{f}}\otimes \bw \left\langle y \right\rangle, dy =0
\right)
\end{equation}
is a model of $N_{(\hat{f},\frac\pi2)}$.
Arguing as in Example~\ref{ex:vaisman}, we can replace
$\Omega^*(N)^{\hat{f}}$ with the CDGA
\begin{equation*}
\left( \overline{ B }\otimes \bw \left\langle \alpha_3  \right\rangle,
d\alpha_3 =-\alpha_1  \wedge \alpha_2  \right),
\end{equation*}
where $\overline{ B }$ is the CDGA with zero differential  and
\begin{equation*}
\overline{ B }_k =
\begin{cases}
\left\langle 1 \right\rangle, & k=0\\
\left\langle \alpha_1 \wedge \alpha_2, \beta_1  \wedge \beta_2 \right\rangle, &
k=2\\
0, & \mbox{otherwise.}
\end{cases}
\end{equation*}
Defining $\overline{ A }=\left( \overline{ B }\otimes \bw \left\langle y
\right\rangle, dy=0 \right)$, we get that~\eqref{omegan} is quasi-isomorphic to
\begin{equation*}
\left( \overline{ A }\otimes \bw\left\langle \alpha_3  \right\rangle, d\alpha_3
 = -\alpha_1  \wedge \alpha_2  \right),
\end{equation*}
which provides an almost formal model for $N_{(\hat{f},\frac\pi2)}$.
\end{example}
\section{Boothby-Wang construction and quasi-Sasakian manifolds.}
\label{BW}
Let $(M,J,g)$ be a compact K\"ahler manifold and $\beta\in \Omega^2(M)$ such that
$i_J \beta =0$ and $[\beta]\in H^2(M,\Z)$.
By \cite[Theorem~2.5]{blair2}, there is a principal circle bundle
$\pi \colon N\to M$ and a $1$-form $\eta$ on $N$ such that $d\eta =
\pi^*\beta$. Denote by $\xi$ the vertical vector field such that $\eta(\xi)=1$.
We can consider $\eta$ as a connection form on the circle bundle $N$, where $\ker\eta$
give subspaces of horizontal vectors.
Denote by $\tilde\pi$ the horizontal lift corresponding to this connection.
Define
\begin{equation*}
\phi:= \tilde\pi J\pi_*.
\end{equation*}
By \cite[Theorem~6]{morimoto63} the structure $(\phi,\xi,\eta)$ is a normal
almost contact structure on $N$.
Now, we define the Riemannian metric $h$ on $N$ by
\begin{equation*}
h (X,Y) = g(\pi_* X, \pi_* Y) + \eta(X) \eta(Y).
\end{equation*}
Now, denote by $\omega$ the symplectic form on $M$. Then
\begin{equation*}
\begin{aligned}
\Phi(X,Y) &= h(X,\phi Y) = g(\pi_* X ,\pi_* \tilde \pi J\pi_* Y)\\& = g (\pi_* X,
J\pi_* Y) = \omega(\pi_* X, \pi_* Y) = (\pi^*\omega)(X,Y).
\end{aligned}
\end{equation*}
Thus
\begin{equation*}
d\Phi = \pi^* d\omega =0.
\end{equation*}
This shows that $(\phi, \xi, \eta, h)$ is a quasi-Sasakian structure on
$N$. Note, that the index of this structure is the largest natural number
$l$ such that $[\beta]^l \not=0$ in~$H^2(M)$.

\label{rankqS}
In Theorem~\ref{classification_qS} we described all nilmanifolds that can admit
a quasi-Sasakian structure. This characterization is based on the index of the
quasi-Sasakian structure. There is also another widely considered invariant of
quasi-Sasakian structures: the rank. We say that a quasi-Sasakian manifold has
\emph{rank} $2p+1$ if
\begin{equation*}
\eta\wedge (d\eta)^{p} \not=0,\quad (d\eta)^{p+1}=0.
\end{equation*}
It is easy to see that the first of the above equations can be replaced by the
seemingly weaker condition
$(d\eta)^p\not=0$. In fact, if $\eta\wedge (d\eta)^p=0$ then also
\begin{equation*}
(d\eta)^p = i_\xi(\eta\wedge (d\eta)^p) = 0,
\end{equation*}
i.e. if $(d\eta)^p\not=0$ then $\eta\wedge (d\eta)^p \not=0$.

It is straightforward to check that if a quasi-Sasakian manifold  has rank $2p+1$ and index
$l$ then $l\le p$.
Therefore if a compact nilmanifold admits a quasi-Sasakian structure of rank
$2p+1$ it is modelled on Lie group $\heis(1,l)\times \R^{2n+1}$ with $l\le p$.
In Section~\ref{quasi-Sasakian}, we showed that for every cocompact subgroup $\Gamma$ of
$G=\heis(1,p)\times \R^{2(n-r)}$ the manifold $\Gamma\backslash G$ admits a quasi-Sasakian structure
of rank $2p+1$. But  the following question remains open: 
\begin{question}
\label{q1}
For which $l<p$ and which cocompact subgroups $\Gamma$ of
\begin{equation*}
G=\heis(1,l)\times
\R^{2(n-l)}
\end{equation*}
the manifold $\Gamma\backslash G$ admits a quasi-Sasakian structure of rank
$2p+1$?
\end{question}
As a first step to answer this question, we show that
 for every $l<p$ there is a cocompact subgroup $\Gamma$
in $G=\heis(1,l)\times \R^{2(n-l)}$ such that the manifold $\Gamma\backslash G$ admits a
quasi-Sasakian structure of rank $2p+1$.

Denote by $\tilde\Gamma$ the subgroup of $\heis(1,l)$ consisting of matrices in $\heis(1,l)$ with integer entries.
 Further we identify $\R^{2(n-l)}$ with
$\C^{n-l}$ considered with its standard K\"ahler structure. Define
\begin{equation*}
\Gamma = \tilde\Gamma \times \left( \Z +\mathbf{i} \Z \right)^{(n-l)} \subset
\heis(1,l)\times \C^{n-l}.
\end{equation*}
Then
\begin{equation*}
\Gamma\backslash G \cong \left( \tilde\Gamma\backslash \heis(1,l)\right)\times \tor_\C^{n-l},
\end{equation*}
where $\tor_\C$ denotes the complex torus $\C/\left( \Z + i\Z \right)$.
Denote by $(\phi,\xi,\eta,h)$ the quasi-Sasakian structure on $\Gamma\backslash G$
constructed in Section~\ref{quasi-Sasakian}.
Then all integral curves of $\xi$ have the same length, and we get the principal
circle bundle
\begin{equation}
\label{fibr}
\xymatrix{
S^1\ \  \ar@{>->}[r] & \Gamma\backslash G \ar@{->>}[d]^{\pi} \\ & \tor_\C^{n}.
}
\end{equation}
Write $\vol$ for the volume form on $\tor_\C$ and denote by $\pr_j$ the
projection from $\tor_\C^n$ to the $j$th component.
Then
\begin{equation*}
d\eta = \pi^*(\pr_1^*\vol + \dots +\pr^*_l \vol).
\end{equation*}
Thus~\eqref{fibr} corresponds to the class $v:=\left[ \pr_1^*\vol + \dots +\pr^*_l
\vol \right]$ in $H^2(\tor_\C^{n},\Z)$.

Let $f$ be any non-constant function in $C^\infty(\tor_\C)$. Then $di_Jdf$ is a
non-zero $2$-form
on $\tor_\C$, since
\begin{equation*}
di_Jdf = (\lap f)\vol.
\end{equation*}
Moreover $i_Jdi_Jdf = d_J^2 f=0$, where $d_J = i_J \circ d + d \circ i_J$ is the differential operator induced by $J$.
Define the $2$-form $\beta$ on $\tor_\C^n$ by
\begin{equation*}
\beta= \pr_1^*\vol + \dots + \pr^*_l \vol + \pr^*_{l+1} di_Jdf + \dots +
\pr^*_p di_Jdf.
\end{equation*}
Our aim is to apply the Boothby-Wang construction described in
this section to the complex manifold $\tor_\C^n$ and the $2$-form
$\beta$. For this we have to check that $i_J \beta =0$ and $[\beta]\in
H^2(\tor_\C^n,\Z) $. The first assertion follows from $i_Jdi_Jdf=0$. For the
second notice that
\begin{equation*}
[\beta] = \left[ \pr_1^*\vol + \dots +\pr^*_l
\vol \right] =v \in H^2(\tor_\C^n,\Z).
\end{equation*}
Denote the resulting quasi-Sasakian manifold by $(N,\phi,\xi',\eta',h')$.
Since the cohomology class of $\beta$ equals to $v$, we see that $N$ is
diffeomorphic to $\Gamma\backslash G$, i.e. $N$ is a nilmanifold.
Write $\pi$ for the projection of $N$ on $\tor^n_\C$. Then
\begin{equation*}
d\eta' = \pi^* \beta.
\end{equation*}
Thus
\begin{equation*}
(d\eta')^p = \pi^* \left( \pr_1^* \vol \wedge \dots \wedge \pr_l^*\vol \wedge
\pr^*_{l+1}di_Jdf \wedge \dots \wedge \pr_{p}^* di_Jdf \right)\not=0
\end{equation*}
and
\begin{equation*}
(d\eta')^{p+1} =0,
\end{equation*}
which shows that $\eta'$ has rank $2p+1$.
Further, we can identify the basic cohomology $H^*_B(N,\xi)$ with
$H^*(\tor_\C^n)$ via $\pi^*$.
Under this identification $[d\eta']_B$ corresponds to $[\beta]=v$. Since
$v^l\not=0$ but $v^{l+1}=0$, we see that $l$ is the index of $N$.

Thus we have
\begin{proposition}
For any $l<p\le n$ there exists a compact $(2n+1)$-dimensional quasi-Sasakian
nilmanifold of rank $2p+1$ with index $l$.
\end{proposition}
Note that in the above example the quasi-Sasakian structure on $N$ can be
considered as a modification of the quasi-Sasakian structure on $\Gamma\backslash G$
constructed in Section~\ref{quasi-Sasakian}. Namely, we have the isomorphism of principal
circle bundles
\begin{equation*}
\xymatrix{
N \ar[rr]^{\cong } \ar@{->>}[dr] && \Gamma\backslash G \ar@{->>}[dl]
\\ & \tor_\C^{n}
}
\end{equation*}
and one can verify that $\xi$ is preserved under this isomorphism.
Thus we can state the following question
\begin{question}
\label{q2a}
For which $\Gamma\subset G = \heis(1,l)\times \R^{2(n-l)}$ there is a
quasi-Sasakian structure on $\Gamma\backslash G$ of rank $2p+1$ with $l<p$ and the same
Reeb vector field as in the quasi-Sasakian structure defined in
Section~\ref{quasi-Sasakian}?
\end{question}
It is clear that if the answer to Question~\ref{q2a} is positive for $\Gamma\backslash G$
then the answer to Question~\ref{q1} is also positive for $\Gamma\backslash G$.

Note that in the above example the form $d\eta'$ on $N$ does not have a
constant rank, in the sense that there are points  $x\in N$ such that
$(d\eta')_x^{p}=0$. In fact, since $di_Jdf$ is an exact $2$-form on a compact
$2$-dimensional manifold $\tor_\C$, there are points $y\in \tor_\C$ such that
$(di_Jdf)_y=0$. Now for any $x\in N$ with $\pr_p\pi x=y$ one gets
$(d\eta')^p_x=0$.

We will say that a quasi-Sasakian manifold $(M,\phi,\xi,\eta,h)$ is of
\emph{constant rank} $2p+1$ if
$(\eta\wedge d\eta)^p_x\not=0$ at every point $x\in M$ and $(d\eta)^{p+1}=0$.
A (partial) answer to the above question should produce an invariant of $\Gamma$ that
guarantees or obstructs the existence of such a structure.

We also can make the same question but with the constant rank in mind.
\begin{question}
\label{q3}
For which cocompact subgroups $\Gamma$ of  $G=\heis(1,l)\times \R^{2(n-l)}$
there is a quasi-Sasakian structure on $\Gamma\backslash G$ of constant rank $2p+1$ for a fixed
integer~\mbox{$p>l$}?.
\end{question}
Differently from the Question~\ref{q1}, in this case we don't know if such
$\Gamma$'s exist. Moreover, we do not even know the answer to the following
\begin{question}
\label{q4}
Is it possible to have a quasi-Sasakian structure on a compact manifold
$M^{2n+1}$
of index $l$ and constant rank $2p+1$ with $l<p< n$?
\end{question}

\bibliography{vaisman}

\def\cprime{$'$}
\providecommand{\bysame}{\leavevmode\hbox to3em{\hrulefill}\thinspace}
\providecommand{\MR}{\relax\ifhmode\unskip\space\fi MR }
\providecommand{\MRhref}[2]{%
  \href{http://www.ams.org/mathscinet-getitem?mr=#1}{#2}
}
\providecommand{\href}[2]{#2}
\begin{thebibliography}{10}

\bibitem{AnUg}
Danielle Angella and Luis Ugarte, \emph{Locally conformal {H}ermitian metrics
  on complex non-{K}\"{a}hler manifolds}, Mediterr. J. Math. \textbf{13}
  (2016), no.~4, 2105--2145. \MR{3530921}

\bibitem{pairs}
Gianluca Bande and Amine Hadjar, \emph{On normal contact pairs}, Internat. J.
  Math. \textbf{21} (2010), no.~6, 737--754. \MR{2658408}

\bibitem{bazzoni}
Giovanni Bazzoni, \emph{Vaisman nilmanifolds}, Bull. Lond. Math. Soc.
  \textbf{49} (2017), no.~5, 824--830. \MR{3742449}

\bibitem{BaMa}
Giovanni Bazzoni and Juan~Carlos Marrero, \emph{On locally conformal symplectic
  manifolds of the first kind}, Bull. Sci. Math. \textbf{143} (2018), 1--57.
  \MR{3767426}

\bibitem{BeGo}
Chal Benson and Carolyn~S. Gordon, \emph{K\"ahler and symplectic structures on
  nilmanifolds}, Topology \textbf{27} (1988), no.~4, 513--518. \MR{976592}

\bibitem{mariza}
Indranil Biswas, Marisa Fern\'andez, Vicente Mu\~noz, and Aleksy Tralle,
  \emph{On formality of {S}asakian manifolds}, J. Topol. \textbf{9} (2016),
  no.~1, 161--180. \MR{3465845}

\bibitem{blairs}
D.~E. Blair, \emph{Geometry of manifolds with structural group {${\mathcal
  U}(n)\times {\mathcal O}(s)$}}, J. Differential Geometry \textbf{4} (1970),
  155--167. \MR{0267501}

\bibitem{blair1}
David~E. Blair, \emph{The theory of quasi-{S}asakian structures}, J.
  Differential Geometry \textbf{1} (1967), 331--345. \MR{0226538}

\bibitem{blair2}
\bysame, \emph{Riemannian geometry of contact and symplectic manifolds}, second
  ed., Progress in Mathematics, vol. 203, Birkh\"auser Boston, Inc., Boston,
  MA, 2010. \MR{2682326}

\bibitem{bicontact}
David~E. Blair, Gerald~D. Ludden, and Kentaro Yano, \emph{Geometry of complex
  manifolds similar to the {C}alabi-{E}ckmann manifolds}, J. Differential
  Geometry \textbf{9} (1974), 263--274. \MR{0346696}

\bibitem{nilsasakian}
Beniamino Cappelletti-Montano, Antonio De~Nicola, Juan~Carlos Marrero, and Ivan
  Yudin, \emph{Sasakian nilmanifolds}, Int. Math. Res. Not. IMRN (2015),
  no.~15, 6648--6660. \MR{3384493}

\bibitem{hltvaisman}
\bysame, \emph{Hard {L}efschetz {T}heorem for {V}aisman manifolds},
  Transactions of the American Mathematical Society \textbf{371} (2019), no.~2,
  755--776. \MR{3885160}

\bibitem{hlt}
Beniamino Cappelletti-Montano, Antonio De~Nicola, and Ivan Yudin, \emph{Hard
  {L}efschetz theorem for {S}asakian manifolds}, J. Differential Geom.
  \textbf{101} (2015), no.~1, 47--66. \MR{3356069}

\bibitem{wolak}
Luis~A. Cordero and Robert~A. Wolak, \emph{Properties of the basic cohomology
  of transversely {K}\"ahler foliations}, Rend. Circ. Mat. Palermo (2)
  \textbf{40} (1991), no.~2, 177--188. \MR{1151583}

\bibitem{goldberg-formula}
Antonio De~Nicola and Ivan Yudin, \emph{Generalized {G}oldberg formula}, Canad.
  Math. Bull. \textbf{59} (2016), 508--520. \MR{3563732}

\bibitem{DGMS}
Pierre Deligne, Phillip Griffiths, John Morgan, and Dennis Sullivan, \emph{Real
  homotopy theory of {K}\"ahler manifolds}, Invent. Math. \textbf{29} (1975),
  no.~3, 245--274. \MR{0382702}

\bibitem{DrOr}
Sorin Dragomir and Liviu Ornea, \emph{Locally conformal {K}\"ahler geometry},
  Progress in Mathematics, vol. 155, Birkh\"auser Boston, Inc., Boston, MA,
  1998. \MR{1481969}

\bibitem{felix_book}
Yves F{\'e}lix, John Oprea, and Daniel Tanr{\'e}, \emph{Algebraic models in
  geometry}, Oxford Graduate Texts in Mathematics, vol.~17, Oxford University
  Press, Oxford, 2008. \MR{2403898}

\bibitem{greub3}
Werner Greub, Stephen Halperin, and Ray Vanstone, \emph{Connections, curvature,
  and cohomology}, Academic Press [Harcourt Brace Jovanovich, Publishers], New
  York-London, 1976, Volume III: Cohomology of principal bundles and
  homogeneous spaces, Pure and Applied Mathematics, Vol. 47-III. \MR{0400275}

\bibitem{hasegawa}
Keizo Hasegawa, \emph{Minimal models of nilmanifolds}, Proc. Amer. Math. Soc.
  \textbf{106} (1989), no.~1, 65--71. \MR{946638}

\bibitem{hodge}
W.~V.~D. Hodge, \emph{The theory and applications of harmonic integrals},
  Cambridge, at the University Press, 1952, 2d ed. \MR{0051571}

\bibitem{kasuya}
Hiroaki Ishida and Hisashi Kasuya, \emph{Transverse {K}\"{a}hler structures on
  central foliations of complex manifolds}, Ann. Mat. Pura Appl. (4)
  \textbf{198} (2019), no.~1, 61--81. \MR{3918619}

\bibitem{malcev}
A.~I. Mal{\cprime}cev, \emph{On a class of homogeneous spaces}, Izvestiya Akad.
  Nauk. SSSR. Ser. Mat. \textbf{13} (1949), 9--32. \MR{0028842}

\bibitem{MaPa}
Juan~Carlos Marrero and Edith Padr{\'o}n, \emph{Compact generalized {H}opf and
  cosymplectic solvmanifolds and the {H}eisenberg group {$H(n,1)$}}, Israel J.
  Math. \textbf{101} (1997), 189--204. \MR{1484876}

\bibitem{morimoto63}
Akihiko Morimoto, \emph{On normal almost contact structures}, J. Math. Soc.
  Japan \textbf{15} (1963), 420--436. \MR{0163245}

\bibitem{lchk}
Liviu Ornea and Paolo Piccinni, \emph{Locally conformal {K}\"ahler structures
  in quaternionic geometry}, Trans. Amer. Math. Soc. \textbf{349} (1997),
  no.~2, 641--655. \MR{1348155}

\bibitem{sullivan}
Dennis Sullivan, \emph{Infinitesimal computations in topology}, Inst. Hautes
  \'Etudes Sci. Publ. Math. (1977), no.~47, 269--331 (1978). \MR{0646078}

\bibitem{tievsky}
Aaron~M Tievsky, \emph{Analogues of k{\"a}hler geometry on sasakian manifolds},
  Ph.D. thesis, Massachusetts Institute of Technology, 2008.

\bibitem{Ug}
Luis Ugarte, \emph{Hermitian structures on six-dimensional nilmanifolds},
  Transform. Groups \textbf{12} (2007), no.~1, 175--202. \MR{2308035}

\bibitem{vaisman_roma}
Izu Vaisman, \emph{Locally conformal {K}\"ahler manifolds with parallel {L}ee
  form}, Rend. Mat. (6) \textbf{12} (1979), no.~2, 263--284. \MR{557668}

\bibitem{lcs-Vai}
\bysame, \emph{Locally conformal symplectic manifolds}, Internat. J. Math.
  Math. Sci. \textbf{8} (1985), no.~3, 521--536. \MR{809073}

\end{thebibliography}
\bibliographystyle{amsplain}

\end{document}